\documentclass[11pt]{article}
\usepackage{amsfonts,amsmath,amssymb,amsthm}
\usepackage{latexsym}
\usepackage{algorithm}
\usepackage{xcolor}
\usepackage{dsfont}
\usepackage{abstract}
\usepackage{graphicx}
\usepackage[margin=1.02in]{geometry}
\usepackage{enumitem}
\usepackage{subcaption}
\usepackage{latexsym,lscape}
\usepackage{algorithm}
\usepackage{algpseudocode}
\usepackage{amsmath}
\usepackage{amsfonts}
\usepackage{amssymb}
\usepackage{booktabs}
\usepackage{mathrsfs}
\usepackage{multirow}
\usepackage{amsthm}
\usepackage{subcaption}
\usepackage{color}
\usepackage{mathdots}
\usepackage{cases}
\usepackage{graphicx}	
\usepackage{rotating}
\usepackage{verbatim}
\usepackage{enumitem}
\usepackage{bm}

\definecolor{pred}{RGB}{148,55,61}
\usepackage[colorlinks,linkcolor=pred,citecolor=pred,anchorcolor=pred]{hyperref}

\newcommand{\ba}{\noindent $\begin{array}}
\newcommand{\ea}{\end{array}$}
\newcommand{\be}{\begin{equation}}
\newcommand{\ee}{\end{equation}}
\newcommand{\bd}{\begin{displaymath}}
\newcommand{\ed}{\end{displaymath}}
\newcommand{\beq}{\begin{eqnarray*}}
\newcommand{\eeq}{\end{eqnarray*}}
\newcommand{\beqn}{\begin{eqnarray}}
\newcommand{\eeqn}{\end{eqnarray}}

\def\hat{\widehat}

\def\[{\begin{equation}}
\def\]{\end{equation}}

\numberwithin{equation}{section}
\newtheorem{theorem}{Theorem}[section]
\newtheorem{proposition}{Proposition}[section]
\newtheorem{lemma}{Lemma}[section]

\title{{\Large\bf A Proximal Point Dual Newton Algorithm for Solving Group Graphical Lasso Problems}}
\date{July, 22, 2020}

\author{
{Yangjing Zhang}\thanks{Department of Mathematics, National University of Singapore, 10 Lower Kent Ridge Road, Singapore 119076 ({\tt zhangyangjing@u.nus.edu}). }
\and Ning Zhang\thanks{(Corresponding author) College of Computer Science and Technology, Dongguan University of Technology, Dongguan 523808, China;
Department of Applied Mathematics, The Hong Kong Polytechnic University, Hung Hom, Hong Kong, China ({\tt ningzhang\_2008@yeah.net}). This author is supported in part by the National Natural Science Foundation of China under Grant 11901083 }
\and Defeng Sun\thanks{Department of Applied Mathematics, The Hong Kong Polytechnic University, Hung Hom,  Hong Kong, China ({\tt defeng.sun@polyu.edu.hk}).
This author is supported in part by Hong  Kong  Research  Grant  Council  under Grant PolyU 153014/18P
}
\and Kim-Chuan Toh\thanks{Department of
Mathematics, and Institute of Operations Research and Analytics, National University of Singapore, 10 Lower Kent Ridge Road, Singapore 119076 ({\tt mattohkc@nus.edu.sg}). This author is supported in part by the Academic Research Fund of the Ministry of Education of Singapore under Grant R-146-000-257-112. }
}

\begin{document}
\maketitle
\vspace{2mm}
\begin{abstract}
Undirected graphical models have been especially popular for learning the conditional independence structure among a large number of variables where the observations are drawn independently and identically from the same distribution. However, many modern statistical problems would involve categorical data or time-varying data, which might follow different but related underlying distributions.  In order to learn a collection of related graphical models simultaneously, various joint graphical models inducing sparsity in graphs and similarity across graphs have been proposed. In this paper, we aim
to propose an implementable proximal point dual Newton algorithm
(PPDNA) for solving the group graphical Lasso model, which encourages a shared pattern of sparsity across graphs. Though the group graphical Lasso regularizer is non-polyhedral, the asymptotic superlinear convergence of our proposed method PPDNA can be obtained  by leveraging on the local Lipschitz continuity of the Karush-Kuhn-Tucker solution mapping associated with the group graphical Lasso model. A variety of numerical experiments on real data sets illustrates that the PPDNA for solving the group graphical Lasso model
can be highly efficient and robust.
\end{abstract}
\smallskip
{\small
\begin{center}
\parbox{0.95\hsize}{{\bf Keywords.}\;
Group Graphical Lasso, Proximal Point Algorithm, Semismooth Newton Method, Lipschitz Continuity}
\end{center}
\begin{center}
\parbox{0.95\hsize}{{\bf AMS Subject Classification.}\; 90C22, 90C25, 90C31, 62J10.}
\end{center}}

\section{Introduction}
\label{sec:intro}
Let $w^{(k)}\in\mathbb{R}^{n_k\times p},\,k=1,2,\ldots,K$ be $K$ given data matrices. For each  $k=1,2,\ldots,K$, the rows of $w^{(k)}$ are observations drawn independently from a Gaussian distribution with mean zero,
and the empirical covariance matrix for $w^{(k)}$ is given by $S^{(k)}=(1/n_k)(w^{(k)})^Tw^{(k)}$. In this paper, we consider the following joint graphical model:
\begin{equation}\label{model-MGL}
\min\limits_{\Theta}~\sum^K_{k=1} \Big(-\log \det \,\Theta^{(k)}+\langle S^{(k)},\Theta^{(k)} \rangle \Big)+ \mathcal{P}(\Theta),
\end{equation}
where $\Theta=\big(\Theta^{(1)},\Theta^{(2)},\dots,\Theta^{(K)}\big)\in \mathbb{S}^p\times\mathbb{S}^p\times\cdots\times\mathbb{S}^p$ is the decision variable, and $\mathcal{P}$ is a convex penalty term that can promote certain desired structure in the decision variable $\Theta$. Throughout this paper, we assume that the solution set to problem \eqref{model-MGL} is nonempty.

If $K=1$ and $\mathcal{P}(\cdot)=\lambda\|\cdot\|_1$, problem \eqref{model-MGL} reduces to the well-known sparse Gaussian graphical model which has been studied by various researchers (e.g.,  \cite{banerjee2008model,bollhofer2019large,friedman2008sparse,hsieh2014quic,rothman2008sparse,yang2013proximal,yuan2006model}).
In many applications, a single Gaussian graphical model is typically enough to capture the conditional independence structure of the random variables. However, in some situations it is more reasonable to fit a collection of such models jointly, due to the similarity or heterogeneity of the data involved. These models for estimating multiple precision matrices jointly are referred to as joint graphical  models in \cite{danaher2014joint}. A scenario where joint graphical models are more suitable than a single graphical model is when the data comes  from several distinct but closely related classes, which share the same collection of variables but differ in terms of the dependency structures. Their dependency graphs can have common edges across a portion of all classes and unique edges restricted to only certain classes. In this case, fitting separate graphical models for distinct classes does not exploit the similarity among the dependency graphs. In contrast, joint estimation of these models could exploit information across different but related classes. In addition to the data from different classes, another scenario that would favor joint graphical models over a single graphical model is when the data contains sequences of multivariate time-stamped observations. Such data might correspond to a series of dependency graphs over time.  Next, we give two practical applications of joint graphical models, which will also be used in our numerical experiments:
\begin{itemize}[topsep=1pt,itemsep=-.6ex,partopsep=1ex,parsep=1ex,leftmargin=3ex]
\item[-] The inference of words relationships from webpages or newsgroups: the webpages from the computer science departments of various universities are classified into several classes: Student, Faculty, Course, Project, etc. The 20 newsgroups are grouped into various topics.
\item[-] The inference of time-varying dependency structures of stocks: the dependency structures among the Standard \& Poor's 500 component stocks might change smoothly over time.
\end{itemize}
In summary, there are two major applications of the joint graphical models: (i) estimating multiple precision matrices jointly for a collection of variables across distinct classes; (ii) inferring the time-varying networks and finding the change-points.

For solving  problem \eqref{model-MGL} with different forms of penalty terms, the alternating direction method of multipliers (ADMM)  has been extensively used; see, e.g., \cite{danaher2014joint,gibberd2017regularized,hallac2017network}.
As we know, the ADMM could be a fast first order method for finding approximate solutions of low or moderate accuracy. However, for attaining superlinear convergence to compute highly accurate solutions, one has to incorporate at least in part the second order information of the problem. Yang et al. \cite{yang2015fused} proposed a proximal Newton-type method, where the subproblem in each iteration can be solved by the nonmonotone spectral projected gradient  method \cite{lu2012augmented,wright2009sparse}, and an active set identification scheme was applied to reduce the cost. Another notable contribution is that a screening rule, which can be combined with any method to reduce the computational cost, was proposed in \cite{yang2015fused}.
However, the second order method in \cite{yang2015fused} is not without drawbacks. Each of its subproblems is a complicated quadratic approximation problem, which generally requires expensive computations. Besides, the inexact proximal Newton-type method proposed in \cite{yang2015fused} has no guarantee of  local linear convergence.
It is worth noting that, in a recent paper related to \cite{yang2015fused}, Yue, Zhou, and So \cite{yue2019family} studied the local convergence rate of a family of  inexact proximal Newton-type methods for solving a class of nonsmooth convex composite optimization problems based on an error bound  condition.
However, it is not clear to us whether the convergence analysis in \cite{yue2019family} can be directly applied to problem \eqref{model-MGL} as the Hessian of the first function in the objective of the problem is not uniformly bounded on its effective domain.
More recently, Zhang et al. \cite{zhang2019efficiency} applied a regularized proximal point algorithm (rPPA) to solve  a fused multiple graphical Lasso (FGL) model and heavily exploited the underlying second order information through the semismooth Newton method when solving the subproblems of the rPPA. Due to the polyhedral property of the FGL regularizer, the rPPA for solving the FGL problem is proven to have an arbitrary linear convergence rate in \cite{zhang2019efficiency}.

Our goal in this paper is to design and analyze an efficient second order information based algorithm with economical implementations and a fast convergence rate for solving problem \eqref{model-MGL} with the following non-polyhedral regularizer, which was referred to as the group graphical Lasso (GGL) regularizer in \cite{danaher2014joint}:
\begin{equation}\label{regularizer-GGL}
\begin{array}{l}
\mathcal{P}(\Theta)\displaystyle = \lambda_1 \sum^K_{k=1} \sum_{i\neq j}|\Theta^{(k)}_{ij}| + \lambda_2 \sum_{i\neq j} \Big( \sum^{K}_{k=1} |{\Theta^{(k)}_{ij}}|^2 \Big)^{1/2},
\end{array}
\end{equation}
where $\lambda_1$ and $\lambda_2$ are positive parameters. We refer to model \eqref{model-MGL} with the regularizer \eqref{regularizer-GGL} as the GGL model. In fact, the GGL regularizer acting on a collection of matrices can be viewed as an extension of the sparse group Lasso regularizer \cite{friedman2010note,Simon2013} acting on a vector. The former can be regarded as the latter if the $(i,j)$-th elements across all $K$ precision matrices are assigned into one group. For $1\leq i,j\leq p$, we let $\Theta_{[ij]} :=[\Theta^{(1)}_{ij};\ldots;\Theta^{(K)}_{ij}]\in\mathbb{R}^K$ be the column vector obtained by taking out the $(i,j)$-th elements across all $K$ matrices $\Theta^{(k)},\,k=1,2,\dots,K$. We can observe that
\begin{equation}\label{regularizer-GGL-v1}
\mathcal{P}(\Theta)=\sum_{i\neq j} \varphi(\Theta_{[ij]})\,\, \hbox{with}\,\, \varphi(x)=\lambda_1\|x\|_1+\lambda_2\|x\|,\,\,\forall\, x\in\mathbb{R}^K,
\end{equation}
where the function $\varphi$ is actually a special sparse group Lasso regularizer.
The first term of the GGL regularizer promotes  sparsity
in the $K$ estimated precision matrices $\Theta^{(k)}$'s.  The zeros in these  precision matrices tend to occur at the same indices  due to the second term of the GGL regularizer. In addition, Figure \ref{fig-GGLReg} illustrates the structure of the decision variable $\Theta$ and the vector belonging to one group $\Theta_{[ij]}$.
\begin{figure}
  \centering
  \includegraphics[width=0.8\textwidth]{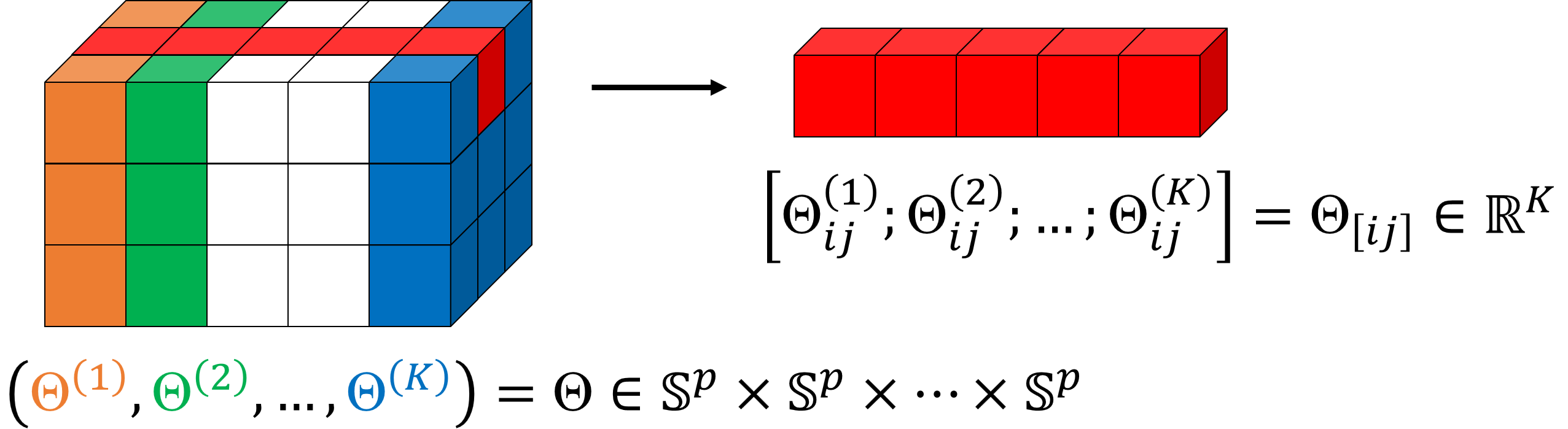}
  \caption{Illustration of $\Theta$ and $\Theta_{[ij]}$. One cube stands for one entry.}\label{fig-GGLReg}
\end{figure}

Inspired by the impressive numerical performance of the rPPA for solving  {the FGL model} \cite{zhang2019efficiency}, we will design a proximal point dual Newton algorithm (PPDNA) for solving the GGL model.
Specifically, a proximal point algorithm (PPA) \cite{rockafellar1976monotone} is applied to the primal formulation of the GGL model, and a superlinearly convergent semismooth Newton   method is designed to solve the dual formulations of the PPA subproblems.
Thanks to the fact that the GGL regularizer is an extension of the sparse group Lasso regularizer, the generalized Jacobian of the proximal mapping of the GGL regularizer  can be characterized based on that of the sparse group Lasso regularizer, where the explicit form was given in \cite{Zhang2018efficient}.
As a result, the former naturally inherits
the structured sparsity (referred to as the second order sparsity) of the latter. Consequently, multiplying a  sparse Hessian matrix by a vector in the semismooth Newton method is reasonably cheap, and one could expect that the superlinearly convergent semismooth Newton method is numerically efficient for solving the PPA subproblems. In addition to achieving low cost in computing the semismooth Newton directions by exploiting the second order sparsity, we also establish the linear convergence guarantee of the PPDNA.

Though the framework of the PPDNA for solving the GGL model is closely related to the rPPA for solving the FGL model \cite{zhang2019efficiency} and the semismooth Newton based augmented Lagrangian method (S{\footnotesize{SNAL}}) for solving the sparse group Lasso problems \cite{Zhang2018efficient},       both the theoretical analysis and numerical implementation should be further investigated owing to the following difficulties of the GGL model. First, unlike the FGL regularizer, the GGL regularizer is a non-polyhedral function and consequently the Lipschitz continuity of the Karush-Kuhn-Tucker (KKT) solution mapping associated with the GGL model is not as straightforward to establish as in \cite{zhang2019efficiency}.
We should mention here that the Lipschitz continuity of the KKT solution mapping plays an important role in establishing the convergence rate of the PPDNA, just as in the case of rPPA and S{\footnotesize{SNAL}}. Second, the subproblem of the PPDNA for solving the GGL model differs from those of the S{\footnotesize{SNAL}} and rPPA which are strongly convex. Therefore, the stopping criteria previously used
in S{\footnotesize{SNAL}} and rPPA are no longer applicable. The main contributions of this paper can be summarized as follows.

\begin{itemize}[topsep=2pt,itemsep=-.6ex,partopsep=1ex,parsep=1ex,leftmargin=3ex]
\item[1.] We prove the Lipschitz continuity of the KKT solution mapping associated with the GGL model, by taking  advantage of the strict convexity of the function $-\log \det(\cdot)$
    in its effective domain, the nonsingularity of  its Jacobian, and Clarke's implicit function theorem \cite{clarke1990optimization,clarke1998nonsmooth}. Consequently,
    the linear convergence of the iterative sequence generated by the PPDNA can be established based on the classical results in \cite{rockafellar1976monotone}. Moreover, by choosing the penalty parameter to be sufficiently large, the PPDNA can be made to attain any desired linear convergence rate. More generally, the Lipschitz continuity of the KKT solution mapping of the model still holds even if the GGL regularizer is replaced by any other convex positively homogeneous function.
\item[2.] We derive a surrogate generalized Jacobian of the proximal mapping of the GGL regularizer. The second order sparsity in the surrogate generalized Jacobian is analyzed in depth and fully exploited in the PPDNA. Therefore, the superlinearly (or even quadratically) convergent semismooth Newton method can solve the PPA subproblems very efficiently since the semismooth Newton directions can be computed cheaply.

\item[3.] We introduce fairly easy-to-check stopping criteria (via the duality theory)
for computing inexact solutions of the PPA subproblems without sacrificing the global or linear convergence of the PPDNA. In fact, the standard stopping criteria adopted by Rockafellar \cite{rockafellar1976monotone} would involve the unknown optimal values of the subproblems, which are not easy to  check unless the objective function is strongly convex with an explicitly given strong convexity parameter.
\end{itemize}
The remaining parts of the paper are organized as follows. Section \ref{sec:pre} presents some definitions and preliminary results, which include the proximal mapping of the GGL regularizer, its generalized Jacobian, the proximal mapping of the log-determinant function and its {derivative}. We analyze in section \ref{sec:kkt} the
Lipschitz continuity of the KKT solution mapping associated with the GGL model, which is
the key property for deriving the linear convergence rate of our proposed algorithm. In section \ref{sec:ppdna}, we propose the PPDNA
for solving the GGL model and investigate its convergence properties. We report the numerical performance of the PPDNA on categorical text data and time-varying stock prices data  in section \ref{sec:numerical} and conclude the paper in section \ref{sec:conclusion}.

\vspace{2mm}
\noindent
{\it Notation.} The following notation will be used in the rest of the paper.
\vspace{1mm}
\begin{itemize}[topsep=1pt,itemsep=-.6ex,partopsep=1ex,parsep=1ex,leftmargin=3ex]
\item $\mathbb{S}^p_+$ ($\mathbb{S}^p_{++}$) denotes the cone of positive semidefinite (definite) matrices in the space of $p\times p$ real symmetric matrices $\mathbb{S}^p$. For any $A,\,B\in \mathbb{S}^p$, we write $A \succeq B$ if $A-B \in \mathbb{S}^p_+$ and $A \succ B$ if $A-B \in \mathbb{S}^p_{++}$. In particular, $A \succeq 0$ ($A\succ 0$) indicates that $A \in \mathbb{S}^p_+$ ($A \in \mathbb{S}^p_{++}$).

\item We let ${\mathbb{Z}}\;({\mathbb{Z}}_+,\;{\mathbb{Z}}_{++})$ be the Cartesian product of $K$ copies of $\mathbb{S}^p \;(\mathbb{S}^p_+,\,\mathbb{S}^p_{++})$.

\item For any matrix $A$, $A_{ij}$ denotes the $(i,j)$-th element of $A$.

\item For any $X:= (X^{(1)},\dots,X^{(K)})\in\mathbb{Z}$, $X_{[ij]} :=[X^{(1)}_{ij};\ldots;X^{(K)}_{ij}]\in\mathbb{R}^K$ denotes the column vector obtained by taking out the $(i,j)$-th elements across all $K$ matrices $X^{(k)},\,k=1,2,\dots,K$.
\item $I_n$  {denotes} the $n\times n$ identity matrix, and $I$  {denotes an} identity matrix or map when the dimension is clear from the  {context.}
\item We use $\lambda_{\max}(\mathcal{A}) $ to denote the largest eigenvalue of a self-adjoint linear operator $\mathcal{A}$.
\item For a given closed convex set $\Omega$ and a vector $x$, we denote the Euclidean projection of $x$ onto $\Omega$ by $\Pi_{\Omega} (x):=\arg\min_{x'\in {\Omega}}\{\|x-x'\|\}$.
\item We denote ceil$(x)$ as the smallest integer greater than or equal to $x\in\mathbb{R}$.
\end{itemize}

\section{Preliminaries}\label{sec:pre}
In this section, we first recall the definition and some relevant properties of the Moreau-Yosida regularization of a proper and closed convex function, which will play an important role in the subsequent theoretical analysis and  algorithmic design. Let $\mathcal{E}$ be a finite dimensional real Hilbert space and $g:\, \mathcal{E}\rightarrow{\mathbb{R}}\cup\{+\infty\}$ be a proper and closed convex function.
The Moreau-Yosida regularization  \cite{moreau1965proximite,yosida1964functional}  of $g$ is defined by
\begin{equation}\label{def-MY}
\Psi_{g}(u):=\min_{u'}\left\{g(u')+\frac{1}{2}\|u'-u\|^2\right\},\,\,\forall\, u\in \mathcal{E}
\end{equation}
and  the proximal mapping of $g$, the unique minimizer of \eqref{def-MY}, is given by
$$
{\rm Prox}_{g}(u) := \arg\min_{u'}\left\{g(u')+\frac{1}{2}\|u'-u\|^2\right\},\,\,\forall\,u\in\mathcal{E}.
$$
One critical property of the Moreau-Yosida regularization is that
$\Psi_{g}(\cdot)$ is a continuously differentiable convex function with the following gradient:
$$
\nabla \Psi_{g} (u) =u-{\rm Prox}_{g}(u),\,\,\forall\,u\in\mathcal{E}.
$$
In addition, the proximal mapping  satisfies
the following Moreau identity \cite[Theorem 31.5]{rockafellar2015convex}:
\begin{equation}\label{moreau-id}
{\rm Prox}_{\sigma g}(u)+\sigma{\rm Prox}_{\sigma^{-1}g^*}(u/\sigma)=u,\,\,\forall\,u\in\mathcal{E},\,\,\sigma>0,
\end{equation}
where $g^*$ is the conjugate function  of $g$ (see e.g.,  \cite{rockafellar2015convex} for its definition).

\subsection{Proximal mapping of the GGL regularizer and its generalized Jacobian}

We investigate in this section the proximal mapping of the GGL regularizer $\mathcal{P}$ in \eqref{regularizer-GGL} and its generalized Jacobian. Recall the function in \eqref{regularizer-GGL-v1}:
\begin{equation*}
\mathcal{P}(\Theta)=\sum_{i\neq j} \varphi(\Theta_{[ij]})\,\, \hbox{with}\,\, \varphi(x)=\lambda_1\|x\|_1+\lambda_2\|x\|,\,\forall\, x\in\mathbb{R}^K.
\end{equation*}
By definition, the proximal mapping of $\mathcal{P}$ is given as follows: for any $X\in \mathbb{Z}$,
\begin{equation}\label{prox-p-GGL}
\begin{array}{l}
 {\rm Prox}_{\mathcal{P}}(X)
 =
 \arg\underset{ \Theta \in \mathbb{Z} }{\min} \left\{ \mathcal{P}(\Theta) + \frac{1}{2}\|\Theta-X\|^2\right\}  \\[3mm]
 = \arg\underset{{\Theta \in \mathbb{Z}}}{\min}\left\{ \sum_{i\neq j} \left\{\varphi(\Theta_{[ij]})+\frac{1}{2}\|\Theta_{[ij]}-X_{[ij]}\|^2\right\} +\frac{1}{2}{\sum_{i}}\|\Theta_{[ii]}-X_{[ii]}\|^2\right\}.
\end{array}
\end{equation}
It is obvious that problem \eqref{prox-p-GGL} is separable for each vector $\Theta_{[ij]}\in\mathbb{R}^K$.
Therefore, for any $i,j\in\{1,2,\ldots,p\}$, the vector $({\rm Prox}_{\mathcal{P}}(X))_{[ij]}$,
 consisting of all entries of $ {\rm Prox}_{\mathcal{P}}(X) $ in the $(i,j)$-th position, is given explicitly by
\begin{equation}\label{rel-prox}
({\rm Prox}_{\mathcal{P}}(X))_{[ij]} =
\begin{cases}
 {\rm Prox}_{\varphi}(X_{[ij]}), & \hbox{\it {if } $ i\neq j$}, \\
 X_{[ii]}, & \hbox{\it {if } $i=j$}.
\end{cases}
\end{equation}
By this equation, one can compute ${\rm Prox}_{\mathcal{P}}$ via performing $p(p-1)/2$ computations of $ {\rm Prox}_{\varphi}$, and this task can be done in parallel.
Parts of the second order information of the underlying problem are contained in the generalized Jacobian of ${\rm Prox}_{\mathcal{P}}$, which  can be characterized by the generalized Jacobian of ${\rm Prox}_{\varphi}$ through using the relationship \eqref{rel-prox} between ${\rm Prox}_{\mathcal{P}}$ and ${\rm Prox}_{\varphi}$.
Fortunately, the generalized Jacobian of ${\rm Prox}_{\varphi}$ has been carefully investigated in \cite{Zhang2018efficient} and has an explicit expression.

Let the multifunction $\widehat{\partial}{\rm Prox}_{\varphi}:\,\mathbb{R}^K\rightrightarrows\mathbb{S}^{K}$ be the generalized Jacobian of $ {\rm Prox}_{\varphi} $. Directly from the formula (10) in \cite{Zhang2018efficient}, the multifunction $\widehat{\partial}{\rm Prox}_{\varphi}$ can be described as follows: for any $u\in\mathbb{R}^K$,
\begin{equation}\label{eq-prox-phi}
\begin{array}{l}
\widehat{\partial}{\rm Prox}_{\varphi}(u) \\[2mm]
= \left\{(I - \Sigma)\Lambda\in\mathbb{S}^K\big|\,
v = {\rm Prox}_{\lambda_1\|\cdot\|_1}(u),\, \Sigma \in \partial\Pi_{\mathbb{B}_{\lambda_2}}(v),\, \Lambda \in \partial {\rm Prox}_{\lambda_1\|\cdot\|_1}(u)
\right\},
\end{array}
\end{equation}
where $ {\mathbb{B}_{\lambda_2}} :=\{v\in{\mathbb{R}^K}\,|\,\|v\|\leq \lambda_2\}$, $\partial\Pi_{\mathbb{B}_{\lambda_2}}$ and $\partial{\rm Prox}_{\lambda_1\|\cdot\|_1}$ are the Clarke generalized Jacobians (see \cite[Definition 2.6.1]{clarke1990optimization} for the definition) of $\Pi_{\mathbb{B}_{\lambda_2}}$ and ${\rm Prox}_{\lambda_1\|\cdot\|_1}$, respectively.
 Therefore, the surrogate generalized Jacobian $\widehat{\partial}{\rm Prox}_{\mathcal{P}}(X):\,\mathbb{Z}\rightrightarrows\mathbb{Z}$ of ${\rm Prox}_{\mathcal{P}}$ at any given $X$ can be described as follows:
\begin{align}\label{jac-prox-p-GGL}
\left\{\begin{array}{l}
\mathcal{W}\in\widehat{\partial}{\rm Prox}_{\mathcal{P}}(X)\hbox{ if and only if $\exists$\;$M^{(ij)}\in\widehat{\partial}{\rm Prox}_{\varphi}(X_{[ij]})$ $\forall\; i < j$},\\[1mm]
\hbox{such that}\ (\mathcal{W}[Y])_{[ij]}=\left\{\begin{array}{ll}
M^{(ij)}Y_{[ij]}, & \hbox{if $ i < j$},\\[0.5mm]
Y_{[ii]}, & \hbox{if $i = j$},\\[0.8mm]
M^{(ji)}Y_{[ij]}, & \hbox{if $ j < i$},
\end{array}
\right.\,\,i,j=1,\ldots,p,\,\,\,\forall\, Y\in\mathbb{Z}.
\end{array}
\right.
\end{align}
The next proposition will explain why $\widehat{\partial}{\rm Prox}_{\mathcal{P}}(X)$ in \eqref{jac-prox-p-GGL} can be treated as the surrogate generalized Jacobian of ${\rm Prox}_{\mathcal{P}}$ at $X$.
Based on \cite[Theorem 3.1]{Zhang2018efficient}, one can easily prove the proposition. We omit the details here.
\begin{proposition}\label{prop-semismooth-GGL}
Let $\mathcal{P}$ be the GGL regularizer defined by \eqref{regularizer-GGL} and $X\in \mathbb{Z}$ be any given element.  The surrogate generalized Jacobian $\widehat{\partial}{\rm Prox}_{\mathcal{P}}(\cdot)$ defined in \eqref{jac-prox-p-GGL} is  nonempty compact valued and upper semicontinuous.
Any element in the set $\widehat{\partial}{\rm Prox}_{\mathcal{P}}(X)$ is a self-adjoint and positive semidefinite operator.
Moreover, we have that, for any $Y\to X$,
$$
{\rm Prox}_{\mathcal{P}}(Y) - {\rm Prox}_{\mathcal{P}}(X) - \mathcal{W}[Y-X] = O(||Y-X||^2),\,\,\forall\,\mathcal{W}\in\widehat{\partial}{\rm Prox}_{\mathcal{P}}(Y).
$$
\end{proposition}

\subsection{Properties of the log-determinant function}
In this subsection, we present some properties on the proximal mapping of the following log-determinant function $h$ and its derivative that are mainly adopted from the papers \cite{wang2010solving,yang2013proximal}:
\begin{equation}\label{def-logdetfct}
h(X):=\left\{\begin{array}{ll}
-\log \det \,X,\,& \hbox{if $ X\in\mathbb{S}^p_{++}$,}\\[1mm]
+\infty, & \hbox{otherwise}.
\end{array}\right.
\end{equation}
Let $\beta > 0$ be given. Define the following scalar functions:
\begin{equation*}
\phi^+_{\beta}(x):=(\sqrt{x^2+4\beta} + x)/2,\,\,
\phi^-_{\beta}(x):=(\sqrt{x^2+4\beta} - x)/2,\,\,\forall\, x\in\mathbb{R}.
\end{equation*}
In addition, for any $A\in \mathbb{S}^p$ with eigenvalue decomposition $A=Q{\rm Diag}(d_1,\ldots,d_p)Q^T$,
we define
\begin{equation*}
\phi^+_{\beta}(A):= Q{\rm Diag}(\phi^+_{\beta}(d_1),\dots,\phi^+_{\beta}(d_p))Q^T,\,\,\,
\phi^-_{\beta}(A):= Q{\rm Diag}(\phi^-_{\beta}(d_1),\dots,\phi^-_{\beta}(d_p))Q^T.
\end{equation*}
One can observe that $\phi^+_{\beta}(A)$ and $\phi^-_{\beta}(A)$ are positive definite for any $A\in\mathbb{S}^p$.
Using the functions defined above, the following two propositions give the proximal mapping of the log-determinant function $h$ and its derivative.
\begin{proposition}\cite[Proposition 2.3]{yang2013proximal}\label{prop:logdet}
Let $h$ be the log-determinant function defined by \eqref{def-logdetfct} and $\beta $ be a positive scalar. Then, {for any $A\in\mathbb{S}^p$}, it holds that
$$
\begin{array}{l}
\phi^+_{\beta}(A) =
{\rm Prox}_{\beta h}(A) = \underset{B \in \mathbb{S}^p_{++}}{\arg\min}
\big\{h(B) + \frac{1}{2\beta}\|B-A\|^2\big\},
\\[4mm]
\Psi_{\beta h}(A)
=\underset{B \in \mathbb{S}^p_{++}}{\min}
\left\{\beta h(B) + \frac{1}{2}\|B-A\|^2\right\}
=-\beta\log\det(\phi^+_{\beta}(A))+\frac{1}{2}\|\phi^-_{\beta}(A)\|^2.
\end{array}
$$
\end{proposition}
\begin{proposition}\label{phiprime}\cite[Lemma 2.1 (b)]{wang2010solving}
Let $\beta$ be a given positive scalar. The function $\phi^+_{\beta}:\,\mathbb{S}^p\to \mathbb{S}^p$ is continuously differentiable.
For any $A\in \mathbb{S}^p$ with eigenvalue decomposition $A=Q{\rm Diag}(d_1,\dots,d_p)Q^T$,
 the derivative $(\phi^+_{\beta})'(A)[B]$ at any $B\in\mathbb{S}^p$ is given by
$$(\phi^+_{\beta})'(A)[B] = Q(\Gamma \odot (Q^T B Q))Q^T,$$
where  $\Gamma\in\mathbb{S}^p$ is defined by
$$
\Gamma_{ij} = \frac{\phi^+_{\beta}(d_i)+\phi^+_{\beta}(d_j)}{({d_i^2+4\beta})^{1/2}+({d_j^2+4\beta})^{1/2}},\,\,i,j=1,2,\dots,p.
$$
\end{proposition}

\section{Lipschitz continuity of the KKT solution mapping}\label{sec:kkt}
In this section, we will prove that the KKT solution mapping associated with the GGL problem 
is Lipschitz continuous. More generally, we emphasize that the Lipschitz continuity of the KKT solution mapping  still holds even if the GGL regularizer is replaced by any other convex positively homogeneous function, since
the key properties we need from the regularizer $\mathcal{P}$ are convexity and positive homogeneity.

The analysis in this section is based on Clarke's implicit function theorem.
For notational convenience, we denote
\begin{equation}\label{def-f}
f(\Theta):= \displaystyle\sum^K_{k=1} h(\Theta^{(k)})+\langle S^{(k)},\Theta^{(k)} \rangle,\,\,\,{\Theta\in\mathbb{Z}.}
\end{equation}
Then the GGL problem \eqref{model-MGL} can be equivalently reformulated as follows:
\begin{equation}\label{model-MGL-C}
\begin{array}{rl}
\min\limits_{\Omega,\Theta}& \displaystyle \left\{f(\Omega)+\mathcal{P}(\Theta)\,|\,\Omega-\Theta =0\right\}.
\end{array}
\end{equation}
The Lagrangian function associated with problem \eqref{model-MGL-C} is given by
$$
\mathcal{L}(\Omega,\Theta,X) := f(\Omega)+\mathcal{P}(\Theta) + \langle\Omega - \Theta,X\rangle,\,\,\,(\Omega,\Theta,X) \in\mathbb{Z}\times \mathbb{Z} \times \mathbb{Z}
$$
and the dual problem  of \eqref{model-MGL-C} is easily shown to be
\begin{equation}\label{model-MGL-Dual}
\max\limits_{X}~
\sum^K_{k=1} \big(\log\det (X^{(k)}+S^{(k)}) + p\big)-\mathcal{P}^*(X).
\end{equation}
In addition, the  KKT system associated with \eqref{model-MGL-C} and \eqref{model-MGL-Dual} is given by
\begin{equation}\label{def-KKT-prox}
\begin{array}{c}
X+{\rm Prox}_{f^*}(\Omega-X)=0,\,\,-X+{\rm Prox}_{\mathcal{P}^*}(\Theta+X)=0,\,\,
\Omega-\Theta=0.
\end{array}
\end{equation}
Since the log-determinant function $h$ is strictly convex and the solution set to problem \eqref{model-MGL} is assumed to be nonempty, problem \eqref{model-MGL-C} has a unique solution. Furthermore, by using \cite[Proposition 4.75]{bonnans2013perturbation} one can easily show that the KKT system \eqref{def-KKT-prox} also has a unique solution, denoted by $(\overline{\Omega},\overline{\Theta},\overline{X})$.

For any given $(U,V,W)\in \mathbb{Z}\times\mathbb{Z}\times\mathbb{Z}$, we consider the following linearly perturbed form of  problem \eqref{model-MGL-C}
\begin{equation}\label{model-MGL-P}
\begin{array}{rl}
\min\limits_{\Omega,\Theta} & \displaystyle f(\Omega)+\mathcal{P}(\Theta)-\langle (U,V),(\Omega,\Theta) \rangle\\[2mm]
{\rm s.t.} & \Omega-\Theta+W =0.
\end{array}
\end{equation}
As in Rockafellar \cite{rockafellar1976augmented}, we define  the following
maximal monotone operator:
$$
\begin{array}{l}
\mathcal{T}_{\mathcal{L}}(\Omega,\Theta,X)\\[2mm]
:=\{(U,V,W)\in \mathbb{Z}\times\mathbb{Z}\times\mathbb{Z}\,|\,(U,V,-W)\in\partial {\mathcal{L}}(\Omega,\Theta,X)\},\,\,\,(\Omega,\Theta,X)\in \mathbb{Z}\times\mathbb{Z}\times\mathbb{Z}.
\end{array}
$$
We also define the KKT solution mapping $\mathcal{S}:\mathbb{Z}\times\mathbb{Z}\times\mathbb{Z} \to \mathbb{Z}\times\mathbb{Z}\times\mathbb{Z}$ as
\begin{equation}\label{def-perb-KKT}
\mathcal{S}(U,V,W):=
\mathcal{T}_{\mathcal{L}}^{-1}(U,V,W)=
\hbox{the set of all KKT points for problem \eqref{model-MGL-P}}.
\end{equation}
Define a mapping $\mathcal{H}:(\mathbb{Z}\times\mathbb{Z}\times\mathbb{Z})
\times(\mathbb{Z}\times\mathbb{Z}\times\mathbb{Z}) \to \mathbb{Z}\times\mathbb{Z}\times\mathbb{Z}$  as follows: for any $(U,V,W)\in \mathbb{Z}\times\mathbb{Z}\times\mathbb{Z}$ and $(\Omega,\Theta,X)\in \mathbb{Z}\times\mathbb{Z}\times\mathbb{Z}$,
\begin{equation}\label{def-fun-H}
\mathcal{H}((U,V,W),(\Omega,\Theta,X))=\left(\begin{array}{c}
X-U+{\rm Prox}_{f^*}(\Omega-X+U)\\[1mm]
-X-V+{\rm Prox}_{\mathcal{P}^*}(\Theta+X+V)\\[1mm]
\Omega-\Theta+W
\end{array}
\right).
\end{equation}
Then it is easy to see that if
$\mathcal{S}(U,V,W)$ is nonempty, then it must be a singleton and satisfies
$\mathcal{H}((U,V,W),\mathcal{S}(U,V,W))=0.$
\begin{lemma}\label{lemma-sub-proxLogdet}
Let $Z\in \mathbb{Z}$ and $f$ be defined by \eqref{def-f}. Then
all  $\mathcal{G}_f\in\partial{\rm Prox}_f(Z)$ and $\mathcal{G}_{f^*}\in\partial{\rm Prox}_{f^*}(Z)$ are self-adjoint and  positive definite with $\lambda_{\max}({\mathcal{G}}_f)<1$ and $\lambda_{\max}(\mathcal{G}_{f^*})<1$.
\end{lemma}
\begin{proof}
The proof can be derived from \cite[Lemma 2.1]{wang2010solving}.
\end{proof}
Since  the GGL regularizer $\mathcal{P}$ defined by \eqref{regularizer-GGL} is positively homogeneous, its conjugate function $\mathcal{P}^*$ is an indicator function of a closed convex set \cite[Example 11.4(a)]{rockafellar2009variational}. Therefore, ${\rm Prox}_{\mathcal{P}^*}$ is the projection onto a closed convex set.
We know further from \cite[Theorem 2.3]{sun2001solving} that for any $Y\in\mathbb{Z}$, any element in $\partial{\rm Prox}_{\mathcal{P}^*}(Y)$ is a self-adjoint operator whose eigenvalues are in the interval $[0,1]$. Thus, by the proof of \cite[Theorem 2.5]{sun2001solving}, we can obtain the following lemma, which will be used in Theorem \ref{th-nonsingular} to analyze the Lipschitz continuity of the KKT solution mapping $\mathcal{S}$ defined by \eqref{def-perb-KKT}.
\begin{lemma}\label{lemma-sub-proxGGL}
Let $Y\in\mathbb{Z}$ and $\mathcal{B}:\mathbb{Z}\to\mathbb{Z}$ be any self-adjoint positive definite operator. Then, for any chosen $\mathcal{G}_{\mathcal{P}^*}\in\partial{\rm Prox}_{\mathcal{P}^*}(Y)$,
the linear operator $I-\mathcal{G}_{\mathcal{P}^*}+\mathcal{G}_{\mathcal{P}^*}\mathcal{B}$ is nonsingular.
\end{lemma}
The next theorem will play an essential role in establishing the linear rate of convergence of our proposed proximal point dual Newton algorithm (PPDNA) for solving the GGL problems in section \ref{linearrate-ppdna}.
\begin{theorem}\label{th-nonsingular}
Let $\mathcal{S}:\mathbb{Z}\times\mathbb{Z}\times\mathbb{Z} \to \mathbb{Z}\times\mathbb{Z}\times\mathbb{Z}$ be the KKT solution mapping defined by \eqref{def-perb-KKT}.
Then the following hold:
\begin{itemize}[topsep=1pt,itemsep=-.6ex,partopsep=1ex,parsep=1ex,leftmargin=5ex]
\item[(a)] Any element in $\partial_{(\Omega,\Theta,X)}\mathcal{H}((0,0,0),(\overline{\Omega},\overline{\Theta},\overline{X}))$ is nonsingular.  Here, we say\\ $\mathcal{G}\in \partial_{(\Omega,\Theta,X)}\mathcal{H}((0,0,0),(\overline{\Omega},\overline{\Theta},\overline{X}))$ if for some linear operator $\mathcal{M}$, it holds that  $(\mathcal{M},\mathcal{G})\in \partial\mathcal{H}((0,0,0),(\overline{\Omega},\overline{\Theta},\overline{X}))$.
\item[(b)] The mapping $\mathcal{S}$ is Lipschitz continuous near the origin; i.e., there exist a neighborhood $\mathcal{N}$ of the origin and a positive scalar $\kappa$ such that $\mathcal{S}(U,V,W)\neq \emptyset$ for any $(U,V,W)\in\mathcal{N}$ and
\begin{equation}\label{def-kappa}
\begin{array}{l}
\|\mathcal{S}(U,V,W)-\mathcal{S}(U',V',W')\|\\[1.5mm]
\leq \kappa\|(U,V,W)-(U',V',W')\|,\,\,\forall\,(U,V,W),(U',V',W')\in \mathcal{N}.
\end{array}
\end{equation}
\end{itemize}
\end{theorem}
\begin{proof}
Since ${\rm Prox}_{\mathcal{P}}$ is directionally differentiable, we know from the Moreau identity \eqref{moreau-id} that ${\rm Prox}_{\mathcal{P}^*}$ is also directionally differentiable. Therefore,  it follows from the chain rule presented in \cite[Lemma 2.1]{sun2006strong} that
for any $\mathcal{G}\in \partial_{(\Omega,\Theta,X)}\mathcal{H}((0,0,0),(\overline{\Omega},\overline{\Theta},\overline{X}))$, there exist
$\mathcal{G}_{f^*}\in \partial{\rm Prox}_{f^*}(\overline{\Omega}-\overline{X})$ and $\mathcal{G}_{\mathcal{P}^*}\in \partial{\rm Prox}_{\mathcal{P}^*}(\overline{\Theta}+\overline{X})$
such that
$$
{\mathcal{G}}(\Delta\Omega,\Delta\Theta,\Delta X)= \left(\begin{array}{c}
\Delta X+\mathcal{G}_{f^*}(\Delta\Omega-\Delta X)\\[1mm]
-\Delta X+\mathcal{G}_{\mathcal{P}^*}(\Delta\Theta+\Delta X)\\[1mm]
\Delta\Omega-\Delta\Theta
\end{array}
\right)
,\,\,\,\forall\, (\Delta\Omega,\Delta\Theta,\Delta X)\in \mathbb{Z}\times\mathbb{Z}\times\mathbb{Z}.
$$
Suppose that there exists  $(\Delta\Omega,\Delta\Theta,\Delta X)\in \mathbb{Z}\times\mathbb{Z}\times\mathbb{Z}$ such that ${\mathcal{G}}(\Delta\Omega,\Delta\Theta,\Delta X)=0$, i.e.,
\begin{equation}\label{def-kkt-eq-per}
\left\{\begin{array}{l}
\Delta X+\mathcal{G}_{f^*}(\Delta\Omega-\Delta X)=0,\\[1mm]
\Delta X-\mathcal{G}_{\mathcal{P}^*}(\Delta\Theta+\Delta X)=0,\\[1mm]
\Delta\Omega-\Delta\Theta=0.
\end{array}
\right.
\end{equation}
It follows from Lemma~\ref{lemma-sub-proxLogdet}  that both  $\mathcal{G}_{f^*}$ and
$\mathcal{B}:=\mathcal{G}^{-1}_{f^*}-I$ are self-adjoint and positive definite.
This, together with \eqref{def-kkt-eq-per}, implies that
\begin{equation}\label{th-sen-eq1}
\Delta\Omega = -\mathcal{B}\Delta X\,\,\,\hbox{and}\,\,\, \mathcal{C}\Delta X=0,
\end{equation}
where $\mathcal{C}:=I-\mathcal{G}_{\mathcal{P}^*}+\mathcal{G}_{\mathcal{P}^*}\mathcal{B}.$
We know from Lemma~\ref{lemma-sub-proxGGL} that $\mathcal{C}$ is nonsingular.
This, together with \eqref{def-kkt-eq-per} and \eqref{th-sen-eq1}, implies that
$$\Delta \Omega=0,\,\,\Delta \Theta =0,\,\,\hbox{and}\,\, \Delta X =0.$$
Therefore, $\mathcal{G}$ is nonsingular, and consequently the statement $(a)$ holds.

The global Lipschitz  continuities of the proximal mappings ${\rm Prox}_{f^*}$ and ${\rm Prox}_{\mathcal{P}^*}$ imply that the  mapping $\mathcal{H}$ defined by \eqref{def-fun-H} is Lipschitz continuous. Therefore, the proof of  $(b)$ can be obtained by $(a)$, the fact that for any $(U,V,W)\in\mathbb{Z}\times\mathbb{Z}\times \mathbb{Z}$, the set $\mathcal{S}(U,V,W)$ must be a singleton if it is nonempty,
and Clarke's implicit function theorem \cite[Page 256]{clarke1990optimization} or \cite[Theorem 3.6]{clarke1998nonsmooth}.
The proof is completed.
\end{proof}

\section{Proximal point dual Newton algorithm}\label{sec:ppdna}

We aim to  develop an implementable proximal point dual Newton algorithm (PPDNA) for solving the GGL problem \eqref{model-MGL-C}. The PPDNA is essentially a proximal point algorithm (PPA) for solving the primal form of the GGL model, and the PPA subproblems
are solved via their corresponding dual problems.
The dual of each subproblem is to maximize a concave function whose gradient is a semismooth function and thus can be solved
by the semismooth Newton  method. We begin this section by introducing the PPA  \cite{rockafellar1976monotone}, i.e., given $\Omega_0,\,\Theta_0\in\mathbb{Z}_{++}$ and $\sigma_0>0$, the updating scheme is given by
\begin{equation}\label{alg-ppa}
\left\{\begin{array}{l}
(\Omega_{t+1},\Theta_{t+1})\approx P_t(\Omega_{t},\Theta_t):= \arg\min\limits_{\Omega,\Theta}~\Phi_{\sigma_t}(\Omega,\Theta),\\
\sigma_{t+1}\uparrow\sigma_{\infty} \leq \infty,\,\,\,t=0,1,\ldots,
\end{array}
\right.
\end{equation}
where
\begin{equation}\label{def-Phi}
\begin{array}{c}
\Phi_{\sigma_t}(\Omega,\Theta):=f(\Omega)+\mathcal{P}(\Theta)+ \frac{1}{2\sigma_t}
\|(\Omega,\Theta)-(\Omega_{t},\Theta_t)\|^2+\delta_{\mathbb{F}}(\Omega,\Theta)
\end{array}
\end{equation}
with $\delta_{\mathbb{F}}$ being the indicator function of the set $\mathbb{F}:=\{(\Omega,\Theta)\in\mathbb{Z}\times\mathbb{Z}\,|\,\Omega-\Theta=0\}$, i.e., $\delta_{\mathbb{F}}(\Omega,\Theta)=0,$ if $(\Omega,\Theta)\in\mathbb{F}$, and $\delta_{\mathbb{F}}(\Omega,\Theta)=\infty,$ if $(\Omega,\Theta)\notin\mathbb{F}.$

We allow for inexactness in the updating scheme \eqref{alg-ppa} and apply the standard criteria proposed by Rockafellar \cite{rockafellar1976monotone} for controlling the inexactness: given nonnegative summable sequences $\{\varepsilon_t\}$ and $\{\gamma_t\}$ such that $\gamma_t<1$ for all $t\geq 0$,
$$
\begin{array}{lr}
\|(\Omega_{t+1},\Theta_{t+1}) - P_t(\Omega_{t},\Theta_t)\|
\;\leq\; \varepsilon_t,& ~~~~~~~{\rm (A)}
\\[3mm]
\|(\Omega_{t+1},\Theta_{t+1}) - P_t(\Omega_{t},\Theta_t)\|
\;\leq\; \gamma_t {\|(\Omega_{t+1},\Theta_{t+1}) - (\Omega_{t},\Theta_t)\|}.& ~~~~~~~{\rm (B)}
\end{array}
$$
Since the exact minimizer $P_t(\Omega_{t},\Theta_t)$ is typically unknown in each iteration, we should introduce practically implementable stopping criteria in place of (A) and (B) in the subsequent analysis.

\subsection{Dual based semismooth Newton method for solving PPA subproblems}
In this section, we aim to design the aforementioned dual based  semismooth Newton method for solving the subproblems of the PPA framework \eqref{alg-ppa}:
\begin{equation}\label{alg-ppa-subopt}
\begin{array}{rl}
\min\limits_{\Omega,\Theta}& f(\Omega)+\mathcal{P}(\Theta)+ \frac{1}{2\sigma_t}
\|(\Omega,\Theta)-(\Omega_{t},\Theta_t)\|^2\\[2mm]
{\rm s.t.} & \Omega-\Theta =0.
\end{array}
\end{equation}
The Lagrangian function associated with problem \eqref{alg-ppa-subopt} is given by
$$
\begin{array}{l}
\mathcal{L}_t(\Omega,\Theta,X):= f(\Omega)+\mathcal{P}(\Theta) + \langle\Omega - \Theta,X\rangle\\[1.5mm]
~~~~~~~~~~~~~~~~~~~~+\frac{1}{2\sigma_t}
\|\Omega-\Omega_t\|^2 + \frac{1}{2\sigma_t}
\|\Theta-\Theta_t\|^2,\,\,\, (\Omega,\Theta,X) \in\mathbb{Z}\times\mathbb{Z}\times\mathbb{Z}
\end{array}
$$
and the Lagrangian dual problem of \eqref{alg-ppa-subopt} is
\begin{equation}\label{model-dual-Reg}
\max\limits_{X} \Big\{\Upsilon_t(X):=\min\limits_{\Omega,\Theta}~\mathcal{L}_t(\Omega,\Theta,X)\Big\}.
\end{equation}
Since we aim to solve problem \eqref{alg-ppa-subopt} via its dual problem \eqref{model-dual-Reg}, a natural idea is to find the explicit expression for the dual objective function $\Upsilon_t$ first. The explicit expression for $\Upsilon_t$
can be obtained from the Moreau--Yosida regularization as follows:
$$
\begin{array}{l}
\Upsilon_t(X)\\[2mm]
= \min\limits_{\Omega}~\big\{f(\Omega)+\langle \Omega,X\rangle+\frac{1}{2\sigma_t}\|\Omega-\Omega_t\|^2\big\}+ \min\limits_{\Theta}~\big\{\mathcal{P}(\Theta)-\langle \Theta,X\rangle+\frac{1}{2\sigma_t}\|\Theta-\Theta_t\|^2\big\}\\[3mm]
= {\displaystyle \sum^{K}_{k=1}} \Big\{\frac{1}{\sigma_t}\Psi_{\sigma_th}\big(\Omega^{(k)}_t-\sigma_t(S^{(k)}+X^{(k)})\big)
-\frac{1}{2\sigma_t}\|\Omega^{(k)}_t-\sigma_t(S^{(k)}+X^{(k)})\|^2\Big\}\\[6mm]
~~~+ {\displaystyle \sum^{K}_{k=1}}\frac{1}{2\sigma_t}\|\Omega^{(k)}_t\|^2+  \frac{1}{\sigma_t}\Psi_{\sigma_t\mathcal{P}}\big(\Theta_t+\sigma_tX\big)-
\frac{1}{2\sigma_t}\|\Theta_t+\sigma_tX\|^2+\frac{1}{2\sigma_t}\|\Theta_t\|^2.
\end{array}
$$
The last equality is achieved when $\Omega^{(k)} = \phi_{\sigma_t}^+ \big(\Omega_t^{(k)} - \sigma_t(S^{(k)} + X^{(k)})\big)$, for $k=1,2,\ldots,K$, and
$\Theta ={\rm Prox}_{\sigma_t\mathcal{P}}({\Theta}_t + \sigma_tX)$.
One can find that $\Upsilon_t$ is continuously differentiable and strongly concave, and
the unique solution to problem \eqref{model-dual-Reg} can be obtained by solving the following nonsmooth system
\begin{equation} \label{newton-equation}
\nabla \Upsilon_t(X) = 0,
\end{equation}
where
$$
\nabla \Upsilon_t(X) = \big(
\phi_{\sigma_t}^+(W^{(1)}_{t}(X)), \ldots, \phi_{\sigma_t}^+(W^{(K)}_{t}(X)) \big)- {\rm Prox}_{\sigma_t \mathcal{P}}\big(V_{t}(X)\big)
$$
with
$$
W^{(k)}_{t}(X):=\Omega^{(k)}_t-\sigma_t(S^{(k)}+X^{(k)}),\,\,k=1,\ldots,K,\,\,\hbox{and}\,\,
V_{t}(X):={\Theta}_t + \sigma_tX.
$$
We can see that if $X=\arg\min_X \Upsilon_t(X) $, then one has that $\Omega=\Theta$ with $\Omega^{(k)} = \phi_{\sigma_t}^+ \big(W^{(k)}_{t}(X)\big)$, for $k=1,2,\ldots,K$,
$\Theta ={\rm Prox}_{\sigma_t\mathcal{P}}(V_t(X))$.
Recall that $\phi_{\sigma_t}^+(\cdot)$ is differentiable and its derivative is given by Proposition~{phiprime}. Thus, the surrogate generalized Jacobian $\widehat{\partial}(\nabla\Upsilon_t)(X){:\,\mathbb{Z}\rightrightarrows\mathbb{Z}}$ of $\nabla\Upsilon_t$ at any $X$ can be defined as follows:
\begin{equation*}
\left\{\begin{array}{l}
\mathcal{V}\in \widehat{\partial}(\nabla\Upsilon_t)(X) \hbox{ if and only if there exists ${\mathcal{W}}\in\widehat{\partial}{\rm Prox}_{\mathcal{P}}(V_{t}(X)/\sigma_t)$ such that}\\[1.5mm]
\hbox{for any $D\in\mathbb{Z}$,}\\[1.5mm]
\mathcal{V}[D] =- \sigma_t \Big((\phi^+_{\sigma_t})'(W_{t}^{(1)}(X))[D^{(1)}],
\dots,(\phi^+_{\sigma_t})'(W_{t}^{(K)}(X))[D^{(K)}]\Big)- \sigma_t {\mathcal{W}}[D].
\end{array}
\right.
\end{equation*}
Based on the surrogate generalized Hessian {$\widehat{\partial}(\nabla\Upsilon_t)(\cdot)$} of $\Upsilon_t$, one can apply  the following dual based semismooth Newton method (Algorithm \ref{alg-ssn}) for solving problem \eqref{model-dual-Reg}
via solving the nonsmooth equation \eqref{newton-equation}. To know more about the local and global methods for  nonsmooth equations, we refer the readers to \cite[Sections 7 \& 8]{facchinei2007finite} and references therein. The main computational cost of Algorithm \ref{alg-ssn} lies in Step~1 for finding the Newton direction. Therefore, we
carefully analyze  the second order sparsity structure in the surrogate generalized Jacobian and fully exploit
the structure to reduce the cost. Due to the computation of $\phi_{\sigma_t}^+(\cdot)$ in $\Upsilon_t$ and $\nabla\Upsilon_t$, the $j$-th iteration of Algorithm \ref{alg-ssn} requires $K m_j $ computations of eigenvalue decompositions.

\begin{algorithm}[h]
\caption{(${\rm DN}(\Omega_{t},\Theta_t,\sigma_t)$) Dual based Semismooth Newton Method. }
\label{alg-ssn}
\vspace{2mm}
Given  $\bar{\eta} \in (0,1)$, $\tau \in (0,1]$,
and $\mu \in (0,1/2)$, $\rho \in (0,1)$. Input $X_{t,0}\in\mathbb{Z}_{++}$.  Iterate the following steps for  $j=0,1,\dots$.\\
{\bf repeat}
\begin{itemize}[topsep=1pt,itemsep=-.1ex,partopsep=1ex,parsep=0.5ex,leftmargin=9ex]
\item[{\sf Step 1.}] (Newton direction) Choose a specific map {$\mathcal{V}_j\in \widehat{\partial}(\nabla\Upsilon_t)(X_{t,j})$}.
 Use the conjugate gradient (CG) method to find an approximate solution $D_j$ to
$$
\begin{array}{c}
{\mathcal{V}_j}[D] = -\nabla {\Upsilon}_t(X_{t,j})
\end{array}
$$
 such that $\| \mathcal{V}_j[D_j]+ \nabla {\Upsilon}_t(X_{t,j})\| \leq \min(\bar{\eta},\|\nabla {\Upsilon}_t(X_{t,j})\|^{1+\tau}).$
\item[{\sf Step 2.}] (Line search) Set $ \alpha_j = \rho^{m_j}$, where $m_j$ is the smallest nonnegative integer $m$ for which
$$
{\Upsilon}_t(X_{t,j} + \rho^m D_j) \geq{\Upsilon}_t(X_{t,j}) + \mu\rho^m \langle \nabla {\Upsilon}_t(X_{t,j}),D_j \rangle,
$$
and set $ X_{t,j+1} = X_{t,j} + \alpha_j D_j $.
\item[{\sf Step 3.}] (Primal-dual iterates) Compute the primal-dual iterates $(\widetilde{\Omega},\widetilde{\Theta},\widetilde{X})$:  
$$
\begin{array}{l}
  \widetilde{\Omega}^{(k)}  = \phi_{\sigma_t}^+\big(\Omega^{(k)}_t-\sigma_t(S^{(k)}+ \widetilde{X}^{(k)})\big),\,\,k=1,\ldots,K, \\[2mm]
\widetilde{\Theta}  =  \widetilde{\Omega},\,\,\,
  \widetilde{X} = X_{t,j+1}.
\end{array}
$$
\end{itemize}
\noindent
{\bf until} $(\widetilde{\Omega},\widetilde{\Theta},\widetilde{X})$ satisfies some given stopping conditions.\\[2mm]
{\bf return} $(\Omega_{t+1},\Theta_{t+1},X_{t+1}) = (\widetilde{\Omega},\widetilde{\Theta},\widetilde{X})$.
\end{algorithm}	

The following proposition states that Algorithm \ref{alg-ssn} for solving the dual of the PPA subproblem \eqref{model-dual-Reg} is globally convergent and locally superlinearly or even quadratically convergent if the parameter $\tau$ is chosen to be $1$.
\begin{proposition}
Let $\{X_{t,j}\}_{j\geq 0}$ be the infinite sequence generated by Algorithm \ref{alg-ssn}. Then  $\{X_{t,j}\}_{j\geq 0}$ converges to the unique optimal solution $\overline{X}_t$ of \eqref{model-dual-Reg}, and the convergence rate is at least superlinear:
$$\|X_{t,j+1}-\overline{X}_t \|=\mathcal{O}(\|X_{t,j} -{\overline{X}_t}\|^{1+\tau}),\,{\tau\in(0,1]}.$$
\end{proposition}
\begin{proof}
Since ${\rm Prox}_{\mathcal{P}}$ is directionally differentiable, it follows from Proposition~\ref{prop-semismooth-GGL} that
${\rm Prox}_{\mathcal{P}}$ is strongly semismooth with respect to the multifunction $\widehat{\partial}{\rm Prox}_{\mathcal{P}}$ in \eqref{jac-prox-p-GGL} (for its definition, see e.g., \cite[Definition 1]{li2017efficiently}). Therefore, the conclusion follows from the {strong concavity}
of $\Upsilon_t(\cdot)$, Proposition \ref{phiprime}, and \cite[Theorem 3]{li2017efficiently}.
\end{proof}

\subsection{Implementable stopping criteria for PPA subproblems}
Due to the lack of explicit forms of the exact solution $P_t(\Omega_{t},\Theta_t)$,
the stopping conditions (A) and (B) need to  be replaced by some implementable conditions.
Since $\Phi_{\sigma_t}$  defined by \eqref{def-Phi} is strongly convex with modulus $1/{2\sigma_t}$, one has the estimate
$$
\begin{array}{c}
\Phi_{\sigma_t}(\Omega_{t+1},\Theta_{t+1}) -  {\inf \Phi_{\sigma_t}(\Omega,\Theta)} \geq \frac{1}{2\sigma_t} \|(\Omega_{t+1},\Theta_{t+1}) - P_t(\Omega_{t},\Theta_t)\|^2,
\end{array}
$$
which implies that
$$
\begin{array}{c}
\|(\Omega_{t+1},\Theta_{t+1}) - P_t(\Omega_{t},\Theta_t)\| \leq \sqrt{2\sigma_t( \Phi_{\sigma_t}(\Omega_{t+1},\Theta_{t+1}) - \inf \Phi_{\sigma_t}(\Omega,\Theta))}.
\end{array}
$$
The unknown value $\inf \Phi_{\sigma_t}(\Omega,\Theta)$
can be replaced by any of its lower bounds converging to it. One choice is to consider the objective value of the dual problem \eqref{model-dual-Reg}. In particular, one has that
$$
\inf \Phi_{\sigma_t}(\Omega,\Theta)= \max \Upsilon_t(X) \geq \Upsilon_t(X_{t,j}),\,\,\forall\,j = 0,1,\ldots
$$
and hence
\begin{equation}\label{stop-cri-Ieq}
\|(\Omega_{t+1},\Theta_{t+1}) - P_t(\Omega_{t},\Theta_t)\| \leq \sqrt{2\sigma_t( \Phi_{\sigma_t}(\Omega_{t+1},\Theta_{t+1}) - \Upsilon_t(X_{t,j}))},\,\,j = 0,1,\ldots.
\end{equation}
Therefore, we can terminate  Algorithm \ref{alg-ssn} if $( \Omega_{t+1}, \Theta_{t+1},X_{t+1})$ satisfies the following conditions:
given nonnegative summable sequences $\{\varepsilon_t\}$ and $\{\gamma_t\}$ such that $\gamma_t<1$ for all $t\geq 0$,
$$
\begin{array}{ll}
\Phi_{\sigma_t}(\Omega_{t+1},\Theta_{t+1}) - \Upsilon_t(X_{t+1})
\;\leq\; {\varepsilon_t^2}/{2\sigma_t},& {\rm (A')}  
\\[3mm]
\Phi_{\sigma_t}(\Omega_{t+1},\Theta_{t+1}) - \Upsilon_t(X_{t+1})
\;\leq\; ({\gamma_t^2}/{2\sigma_t}) \|(\Omega_{t+1},\Theta_{t+1}) - (\Omega_{t},\Theta_{t})\|^2.  & {\rm (B')} 
\end{array}
$$

\subsection{Linear rate convergence of PPDNA}\label{linearrate-ppdna}
Now, we are ready to formally present the promised PPDNA for solving problem \eqref{model-MGL-C}.
\begin{algorithm}[H]
\caption{(PPDNA) Proximal Point Dual based Newton Algorithm.}\label{alg-ppdna}
\vspace{1mm}
Input $\Omega_0,\,\Theta_0\in\mathbb{Z}_{++}$ and $\sigma_0>0$. Iterate the following steps for $t=0,1,\dots$:
\begin{description}
 \item[{\sf Step 1.}] Apply Algorithm \ref{alg-ssn} to obtain
$$
\begin{array}{c}
(\Omega_{t+1},\Theta_{t+1},X_{t+1})={\rm DN}(\Omega_{t},\Theta_t,\sigma_t)
\end{array}
$$
{under stopping criterion ${\rm (A')}$ or ${\rm (B')}$}.
\vspace{0.5mm}
 \item[{\sf Step 2.}]Update $\sigma_{t+1}\uparrow\sigma_{\infty} \leq \infty$.
\end{description}
\end{algorithm}	
Along the line of Rockafellar's works \cite{rockafellar1976augmented,rockafellar1976monotone}, the local linear convergence rate of the primal and dual iterative sequences generated by the PPA can be guaranteed by the Lipschitz continuity of the KKT solution mapping near the origin under proper stopping criteria of the PPA subproblems. However, the Lipschitz property of the KKT solution mapping requires the uniqueness of the KKT point, and this property is not straightforward to establish when the regularizer $\mathcal{P}$ is not a piecewise linear-quadratic function.
As the property to ensure the linear convergence rate, especially the uniqueness assumption, is too restrictive to hold, Luque \cite{luque1984asymptotic} extended the results and proved the local linear convergence of the PPA under the local upper Lipschitz continuity (see e.g., \cite[p.~208]{robinson1981some} for the definition) of the KKT solution mapping at the origin \cite[p.~387]{cui2019r}.
The local upper Lipschitz continuity condition does not make the assumption on the uniqueness of the solution. However, the local upper Lipschitz continuity property may not hold when the KKT solution mapping is not piecewise polyhedral. Fortunately, for our GGL model, the strict convexity of the log-determinant function guarantees the uniqueness of the solution, and we prove in Theorem \ref{th-nonsingular} that the KKT solution mapping $\mathcal{S}$ of the GGL model (defined by \eqref{def-perb-KKT}) is Lipschitz continuous near the origin by taking advantage of the nice properties of the log-determinant function and Clarke's implicit function theorem. Therefore, the local linear convergence rate of the PPDNA can be obtained via  the classical results by Rockafellar.
The convergence results of Algorithm~\ref{alg-ppdna} for solving problem \eqref{model-MGL-C} are presented below.
\begin{theorem}\label{PPDNA-linear}
Let $\{(\Omega_{t},\Theta_{t},X_{t}) \}_{t\geq 0}$ be an infinite sequence generated by Algorithm~\ref{alg-ppdna} under stopping criterion ${\rm (A')}$. Then the sequence $\{(\Omega_{t},\Theta_{t})\}_{t\geq 0}$ converges to the unique solution $(\overline{\Omega},\overline{\Theta})$ of \eqref{model-MGL-C}, and the sequence $\{X_t\}_{t\geq 0}$ converges to the unique solution $\overline{X}$ of \eqref{model-MGL-Dual}.
Furthermore, if {both criteria ${\rm (A')}$ and ${\rm (B')}$ are} executed in Algorithm~\ref{alg-ppdna}, then there exists $\bar{t}\geq 0$ such that for all $t\geq \bar{t}$, {one has that}
\begin{equation*}
\begin{array}{c}
\|(\Omega_{t+1},\Theta_{t+1},X_{t+1})-(\overline{\Omega},\overline{\Theta},\overline{X})\|\leq \varrho_t\|(\Omega_{t},\Theta_{t},X_t)-(\overline{\Omega},\overline{\Theta},\overline{X})\|,
\end{array}
\end{equation*}
where the convergence rate is given by
$$
1>\varrho_t:=[\kappa(\kappa^2+\sigma^2_t)^{-1/2}+\gamma_t]/(1-\gamma_t)\rightarrow \varrho_{\infty}=\kappa(\kappa^2+\sigma^2_{\infty})^{-1/2}\,\,(\varrho_{\infty}=0\,\hbox{if}\,\,\sigma_{\infty}=\infty)
$$
and the parameter $\kappa$ is from \eqref{def-kappa}.
\end{theorem}
\begin{proof}
The global convergence of Algorithm~\ref{alg-ppdna} can be obtained from \eqref{stop-cri-Ieq}, \cite[Theorem 1]{rockafellar1976monotone}, and the uniqueness of the KKT point.
The linear rate of convergence can be derived from  \eqref{stop-cri-Ieq},  Theorem~\ref{th-nonsingular}${\rm (b)}$, and  \cite[Theorem 2]{rockafellar1976monotone}. The proof is completed.
\end{proof}

\subsection{Extensions of PPDNA}
Although the theoretical analysis and the algorithmic design presented in section 3 and section 4 focus on the GGL regularizer,   these results can also be applied
to the joint graphical model \eqref{model-MGL} with a different regularizer satisfying the following conditions:
\begin{itemize}[topsep=1pt,itemsep=-.6ex,partopsep=1ex,parsep=1ex,leftmargin=5ex]
\item[(a)] the regularizer is convex and positively homogenous, (e.g., a norm function);
\item[(b)] the proximal mapping associated with the regularizer can be efficiently computed and its surrogate generalized Jacobian can be explicitly characterized.
\end{itemize}
For example, we can show that both the pairwise fused graphical Lasso regularizer~\cite[Equation (5)]{danaher2014joint} and the sequential fused graphical Lasso regularizer~\cite[Formula (2.2)]{yang2015fused} satisfy the conditions (a) and (b). More specifically,
let $\lambda_1$ and $\lambda_2$ be positive parameters. The
pairwise fused graphical Lasso regularizer and sequential fused graphical Lasso regularizer are given as follows:
\begin{equation}\label{regularizer-PFGL}
\begin{array}{l}
\mathcal{P}_1(\Theta)\displaystyle = \lambda_1 \sum^K_{k=1} \sum_{i\neq j}|\Theta^{(k)}_{ij}| + \lambda_2 \sum_{k<k'} \sum_{i\neq j} |{\Theta^{(k)}_{ij}}-{\Theta^{(k')}_{ij}}|=\sum_{i\neq j} \varphi_1(\Theta_{[ij]}),
\end{array}
\end{equation}
where $\varphi_1(x)= \lambda_1\|x\|_1 + \lambda_2 \sum_{k < k'}|x_k -x_{k'}|,\,\,x\in\mathbb{R}^K$, and
\begin{equation}\label{regularizer-SFGL}
\begin{array}{l}
\mathcal{P}_2(\Theta)\displaystyle = \lambda_1 \sum^K_{k=1} \sum_{i\neq j}|\Theta^{(k)}_{ij}| + \lambda_2 \sum^K_{k=2} \sum_{i\neq j} |{\Theta^{(k)}_{ij}}-{\Theta^{(k-1)}_{ij}}|=\sum_{i\neq j}{\varphi_2}(\Theta_{[ij]}),
\end{array}
\end{equation}
where $\varphi_2(x)= \lambda_1\|x\|_1 + \lambda_2 \sum^K_{k=2}|x_{k}-x_{k-1}|,\,\,x\in\mathbb{R}^K.$

By applying the same procedure as in section 2.1 for the GGL regularizer, we can obtain
the proximal mappings associated with $\mathcal{P}_i,\,i=1,2$ and their surrogate generalized Jacobians from that of the clustered Lasso regularizer \cite{lin2019efficient} and the fused lasso regularizer \cite{li2017efficiently}, respectively. Therefore, we can apply our PPDNA framework for solving the joint graphical model with a
 different regularizer given by either \eqref{regularizer-PFGL} or \eqref{regularizer-SFGL}.

In addition to the direct extensions to the two regularizers above, the  PPDNA framework  is also applicable to joint graphical models with
other regularizers discussed in \cite{hallac2017network}. More specifically,
the regularizers  in \cite{hallac2017network} have the following form:
$$
\begin{array}{l}
\mathcal{P}(\Theta)\displaystyle = \mathcal{Q}_1(\Theta)+\mathcal{Q}_2(\Theta),
\end{array}
$$
where $
\mathcal{Q}_1(\Theta):=\lambda_1 \sum^K_{k=1} \sum_{i\neq j}|\Theta^{(k)}_{ij}|,\,\, \mathcal{Q}_2(\Theta):= \lambda_2 \sum^K_{k=2}\psi
(\Theta^{(k)}-\Theta^{(k-1)}).$
All the choices of the penalty function $\psi$ in \cite[Section 2.1]{hallac2017network} can ensure condition (a) except for the Laplcacian penalty $\psi(\cdot) = \|\cdot\|^2$. Therefore, except for the case of the Laplacian penalty, the PPDNA framework can be directly applied once condition (b) holds. For the exceptional case, we may slightly modify our framework. Specifically, each iteration should be modified as follows:
given $\Omega_0,\,\Theta_0,\,\Lambda_0\in\mathbb{Z}_{++}$ and $\sigma_0>0$, the updating scheme is given by
\begin{equation}\label{alg-ppa-2}
\left\{\begin{array}{l}
(\Omega_{t+1},\Theta_{t+1},\Lambda_{t+1})\approx P_t(\Omega_{t},\Theta_t,\Lambda_{t+1}):= {\arg\min\limits_{\Omega,\Theta,\Lambda}~\Phi_{\sigma_t}(\Omega,\Theta,\Lambda)},\\
\sigma_{t+1}\uparrow\sigma_{\infty} \leq \infty,\,\,\,t=0,1,\ldots,
\end{array}
\right.
\end{equation}
where
\[
\begin{array}{c}
\Phi_{\sigma_t}(\Omega,\Theta,\Lambda):=f(\Omega)+\mathcal{Q}_1(\Theta)+\mathcal{Q}_2(\Lambda)+ \frac{1}{2\sigma_t}
\|(\Omega,\Theta,\Lambda)-(\Omega_{t},\Theta_t,\Lambda_t)\|^2+\delta_{\mathbb{F}}(\Omega,\Theta,\Lambda),
\end{array}
\]
with $\delta_{\mathbb{F}}$ being the indicator function of the set
$$\mathbb{F}:=\{(\Omega,\Theta,\Lambda)\in\mathbb{Z}\times\mathbb{Z}\,|\,\Omega-\Theta=0,\, \Lambda-\Theta=0\}.$$
Then the resulting modified PPDNA can be obtained by using arguments similar to those in section 4.
But we should  mention that further investigation will be necessary to overcome the underlying difficulty
that the dual of the subproblems of  \eqref{alg-ppa-2} may not necessarily be strongly convex. We leave this part as our future
research topic.

\section{Numerical results}\label{sec:numerical}
In this section, we evaluate the performance of the PPDNA in comparison with the ADMM
and the proximal Newton-type method implemented in the work \cite{yang2015fused}, which is referred to as MGL\footnote{The solver is available at {\tt http://senyang.info/}.}.
All the experiments are performed in {\sc Matlab} (version 9.7) on a Windows workstation (24-core, Intel Xeon E5-2680 @ 2.50GHz, 128 GB of RAM).

\subsection{Implementation of ADMM}\label{sec-admm}

In this section, we briefly describe the ADMM for solving the dual problem \eqref{model-MGL-Dual}, which can be equivalently written as follows:
\begin{equation}\label{model-ADMM}
\begin{array}{cl}
\min\limits_{X,\,Z} & \displaystyle \left\{ \sum^K_{k=1} h(Z^{(k)})+ \mathcal{P}^*(X) \,\big|\, Z-X=S\right\}.
\end{array}
\end{equation}
Given a parameter $\sigma>0$,
the augmented Lagrangian function associated with \eqref{model-ADMM} is defined by
$$
\widehat{\mathcal{L}}_{\sigma} (X,Z,\Theta)=\sum^K_{k=1} h (Z^{(k)})+ \mathcal{P}^*(X)+\langle Z-X-S,\,\Theta\rangle+\displaystyle\frac{\sigma}{2}\|Z-X-S\|^2
$$
and the KKT optimality conditions are
\begin{equation*}
\Theta-{\rm Prox}_{\mathcal{P}}(\Theta+X)=0,\,\,
 Z - X - S=0,\,\,
Z^{(k)} - {\rm Prox}_h(Z^{(k)}-\Theta^{(k)})=0,\,\,\,k=1,\ldots,K.
\end{equation*}
The iterative scheme of the ADMM for problem \eqref{model-ADMM} can be described as follows: given $\tau\in(0,(1+\sqrt{5})/2)$ and an initial point $(X_0,Z_0,\Theta_0)\in\mathbb{Z}_{++}\times\mathbb{Z}_{++}\times\mathbb{Z}_{++}$,  the $t$-th iteration is given by
\begin{equation}\label{alg-ADMM}
\left\{\begin{array}{l}
X_{t+1}=
\arg\min_{X}~\widehat{\mathcal{L}}_{\sigma} (X,Z_t,\Theta_t),
\\[2mm]
Z_{t+1}=
\arg\min_{Z}~\widehat{\mathcal{L}}_{\sigma} (X_{t+1},Z,\Theta_t),
\\[2mm]
\Theta_{t+1}=
\Theta_t + \tau\sigma(Z_{t+1}-X_{t+1}-S),
\end{array}
\right.
\end{equation}
where $X_{t+1}$ can be updated by $X_{t+1}=(Z_{t}+ {\sigma}^{-1}\Theta_t-S)-{\rm Prox}_{\mathcal{P}}(Z_{t}+ {\sigma}^{-1}\Theta_t-S)$, and
$Z_{t+1} = (Z^{(1)}_{t+1},\dots,Z^{(K)}_{t+1})$ can be updated by
$Z^{(k)}_{t+1} = \phi^+_{{\sigma}^{-1}}\big(X^{(k)}_{t}-\frac{1}{\sigma}\Theta^{(k)}_t +S^{(k)} \big)$, $k = 1,\ldots,K$.

In the practical implementation, we tuned the parameter $\sigma$ according to the progress of primal and dual feasibilities (see e.g., \cite[Section 4.4]{lam2018fast}) and used a larger step-length  $\tau$ of $1.618$.  These two techniques can empirically accelerate the convergence speed.
It is worth noting that the ADMM implemented by Yang et al. \cite{yang2015fused} used a fixed penalty parameter $\sigma$ and the step-length $\tau=1$.

\subsection{Settings of experiments}\label{sec-stopcond}
{The experimental settings are the same as those in \cite[Section 4]{zhang2019efficiency}.}
We adopt the stopping criteria of PPDNA, ADMM and MGL as below. Let $\epsilon>0$ be a given tolerance. It is set as $ 10^{-6}$ in the following experiments.
\vspace{1mm}
\begin{itemize}[topsep=2pt,itemsep=-.6ex,partopsep=1ex,parsep=1ex,leftmargin=4ex]
\item The PPDNA is terminated if $\eta_P\leq \epsilon$, where
$$
\eta_P:=\max \left\{
\frac{\|\Theta-{\rm Prox}_{\mathcal{P}}(\Theta + X )\|}{1+\|\Theta\|},\,
\frac{\|\Theta - \Omega \|}{1+\|\Theta\|},\,
\frac{\|\Omega-{\rm Prox}^K_{h}(\Omega-S-X)\|}{1+\|\Omega\|}
\right\}
$$
with
$$
{\rm Prox}^K_h(\Theta):= (  {\rm Prox}_h(\Theta^{(1)}),  \dots,  {\rm Prox}_h(\Theta^{(K)})  )\in\mathbb{Z},\,\,\Theta\in\mathbb{Z}.
$$
\item The ADMM is terminated when $\eta_A\leq \epsilon$ or $20000$ iterations are taken, where
$$
\eta_A:=\max \left\{
\frac{\|\Theta-{\rm Prox}_{\mathcal{P}}(\Theta + X)\|}{1+\|\Theta\|},\,
\frac{\|Z - X- S \|}{1+\|S \|},\,
\frac{\|Z-{\rm Prox}^K_h(Z-\Theta)\|}{1+\|Z\|}
\right\}.
$$
\item The MGL is
terminated when the relative difference of its objective value with respect to the primal objective value obtained by the PPDNA is smaller than the given tolerance $\epsilon$  or the relative duality gap achieved by the PPDNA, i.e.,
$$
\Delta_M:= \frac{{\rm pobj}_{M} - {\rm pobj}_{P}}{1+|{\rm pobj}_{M}|+|{\rm pobj}_{P}|} < \max\{\epsilon, {\rm relgap}_P\},
$$
where ${\rm pobj}_P$, ${\rm pobj}_{M}$, and ${\rm relgap}_P$ are the objective values obtained by the PPDNA, the MGL, and the relative duality gap attained by the PPDNA.
\end{itemize}
It is worth mentioning that we adopt a warm-starting technique  in the initial stage of the PPDNA, instead of starting it from scratch. The warm-starting  procedure consists of first running the ADMM (with identity matrices as the starting point) for a fixed number of iterations (3000 steps in our experiments) or up to a given tolerance ($100 \epsilon$ in our experiments), and  then using the resulting approximate solution as  an initial point to warm-start the PPDNA. This idea is greatly motivated by two facts: 1) the ADMM can generate a solution of low to medium accuracy efficiently and might become slow when higher accuracy is required; 2) our algorithm PPDNA has been proven to be locally linearly convergent.
Therefore, the warm-starting technique can integrate the advantages of both ADMM and PPDNA.

We set the initial parameter in the stopping criterion ${\rm (A')}$  to be $\varepsilon_0=0.5$, and decrease it by a ratio $\varsigma>1$, i.e., $\varepsilon_{k+1} = \varepsilon_k/\varsigma$.
Likewise,  the parameter $\gamma_k$ in  the stopping criterion ${\rm (B')}$ is updated in the same fashion as that for $\varepsilon_k$.
For the parameters in Algorithm~\ref{alg-ssn}, we  simply set $\bar{\eta}=0.1$ and $\tau\in [0.1,0.2]$ according to \cite{zhao2010newton}. For the step on line search, we set $\mu=10^{-4}$ and $\rho=0.5$.

\subsection{Descriptions of datasets}
In this part, we describe the datasets which will be used later.
Since these datasets have been discussed in \cite{zhang2019efficiency}, we briefly review them for the ease of reading:
\vspace{1mm}
\begin{itemize}[topsep=1pt,itemsep=-.6ex,partopsep=1ex,parsep=1ex,leftmargin=4ex]
\item {\it  University webpages data set\footnote{{\tt http://ana.cachopo.org/datasets-for-single-label-text-categorization}}:} The original data was collected from computer science departments of various universities in 1997, manually classified into seven different classes: Student, Faculty, Course, Project, Staff, Department, and Other.
The data we use, consisting of the first four classes,  is preprocessed  by stemming techniques \cite{ana2007improving}. Two thirds of the pages were randomly chosen for training
({\it{Webtrain}}) and the remaining third for testing ({\em{Webtest}}).

\item {\it 20 newsgroups data set\footnote{\tt {http://qwone.com/$\sim$jason/20Newsgroups/}}:} This data set has 20 topics of newsgroup documents, and some of the topics are closely related to each other, while others are highly unrelated. 
    Four subgroups are named as {\it NGcomp}, {\it NGrec}, {\it NGsci}, and {\it NGtalk} accordingly and will be used in our experiments.

\item {\it  SPX500 component stocks\footnote{\tt{www.yahoo.com}}:} This data set contains the  daily returns of Standard \& Poor's 500 ({\it SPX500}) constituents from 2004 to 2014. We also test on extracted data from 2004 to 2006.
\end{itemize}

\subsection{Performance of PPDNA}
In this part, we first give an elementary report of the effectiveness of the GGL model  on synthetic nearest-neighbor networks generated by the mechanism in \cite{li2006gradient}. Second, we illustrate numerically  the local linear convergence of the PPDNA for solving two representative instances, in correspondence with Theorem \ref{PPDNA-linear} which shows theoretically the local linear convergence of the PPDNA.

\subsubsection{Synthetic data: nearest-neighbor networks}\label{sec-nn}
In this example, we choose $p=500$ and $K=3$. The synthetic precision  matrices, denoted as $\Sigma^{(k)},\,k=1,\dots,K$, are generated as follows. We first generate $p$ points on a unit square randomly, calculate their pairwise distances, and identify $5$ nearest neighbors of each point. The nearest-neighbor network is then obtained by linking any two points that are $5$ nearest neighbors of each other, and we denote the number of its edges as $N$. Subsequently, we obtain each $\Sigma^{(k)},\,k=1,2,3$ by adding extra edges to the common nearest-neighbor network. For each $k$, a pair of symmetric zero elements is randomly selected from the nearest-neighbor network and replaced with a value uniformly drawn from the interval $[-1,-0.5]\cup[0.5,1]$. $\Sigma^{(k)}$ is obtained after  this procedure is
repeated ceil$(N/4)$  times. We find in our simulation that the true number of edges in the three networks is $3690$.
Given the precision matrices, we draw $10000$ samples from each Gaussian distribution $\mathcal{N}_p(0,(\Sigma^{(k)})^{-1})$ to compute the sample covariance matrices.
Next we specify the tuning parameters $\lambda_1$ and $\lambda_2$. Following \cite{danaher2014joint}, we reparameterize $\lambda_1$ and $\lambda_2$ in order to separate the regularization for ``sparsity'' and for ``similarity'' since both parameters contribute to sparsity: $\lambda_1$ drives individual network edges to zero whereas $\lambda_2$ drives network edges to zero across all $K$ network estimates at the same time. We reparameterize them  in terms of
$
w_1 = \lambda_1 + \frac{1}{\sqrt{2}}\lambda_2,\,\,w_2 =  \frac{1}{\sqrt{2}}\lambda_2/(\lambda_1 + \frac{1}{\sqrt{2}}\lambda_2),
$
which are found in  \cite{danaher2014joint} to reflect the levels of sparsity and similarity regularization and are called   the sparsity and similarity control parameters, respectively.  In order to show the diversity of sparsity in our experiments, we change $w_1$  with $w_2$  fixed.
Figure \ref{fig-nn} characterizes the relative abilities of the GGL model  to recover the network structures and to detect change-points.

\begin{figure}[!h]
\centering
\begin{subfigure}[b]{.32\linewidth}
\centering
\includegraphics[width=1.0\textwidth]{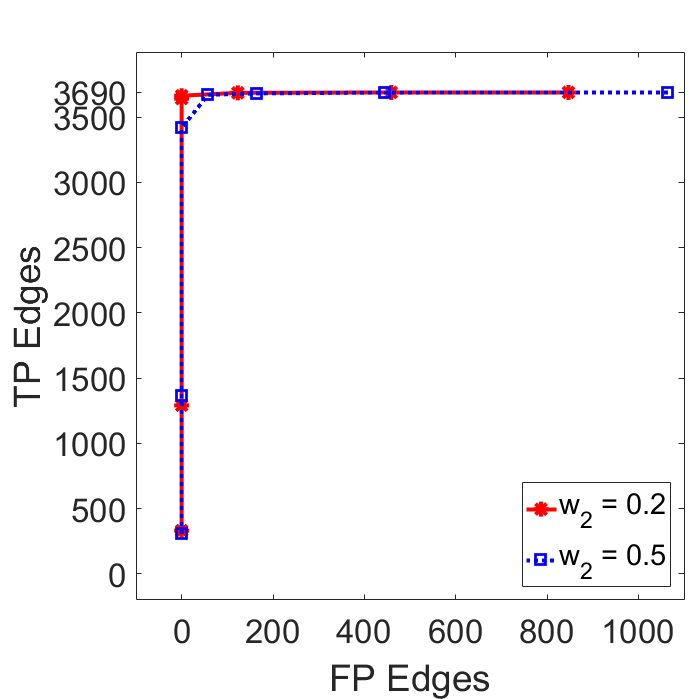}
\caption{}\label{figsub-4}
\end{subfigure}
\begin{subfigure}[b]{.32\linewidth}
\centering
\includegraphics[width=1.0\textwidth]{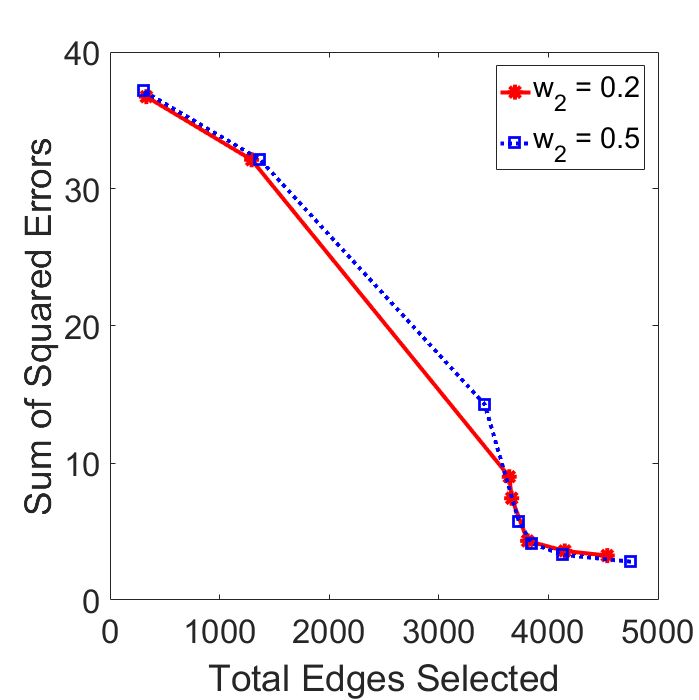}
\caption{}\label{figsub-5}
\end{subfigure}
\begin{subfigure}[b]{.32\linewidth}
\centering
\includegraphics[width=1.0\textwidth]{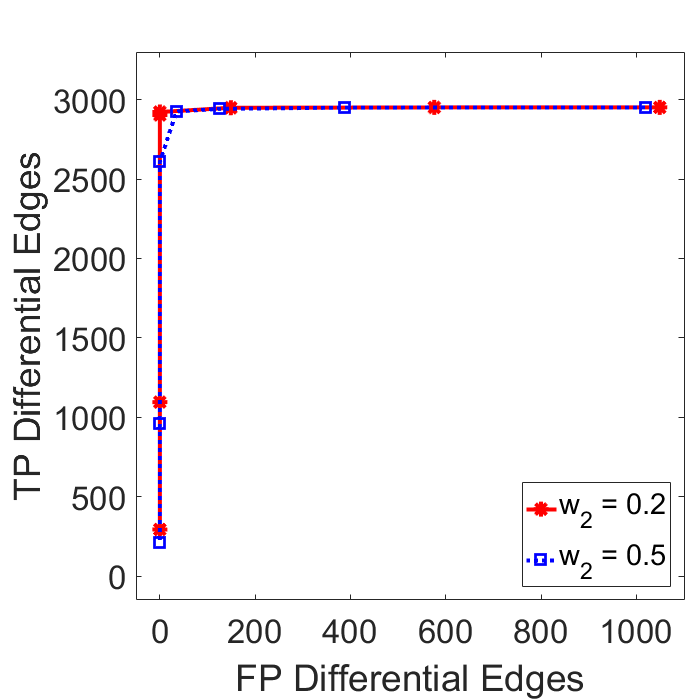}
\caption{}\label{figsub-6}
\end{subfigure}
\caption{Performances of  the GGL   model  on nearest-neighbor networks ($p=500$, $K=3$). (a) the  number of true positive edges versus the number of false positive edges; (b) the  sum of squared errors in edge values versus the total number of edges selected; (c) the  number of true positive differential edges versus  the number of false positive differential edges.}
  \label{fig-nn}
\end{figure}

Figure \ref{figsub-4} displays the number of true positive (TP) edges selected  against the number of false positive (FP) edges. We say that an edge $(i,j)$ is selected in the estimate $\overline{\Theta}^{(k)}$ if $\overline{\Theta}^{(k)}_{ij} \neq 0$, and  the edge is true if $\Sigma^{(k)}_{ij} \neq 0$ and false if $\Sigma^{(k)}_{ij} = 0$.
We can see that the model with $w_2=0.2$ can recover almost all of the TP edges without FP edges. This suggests that the GGL model is effective for recovering the edges in the nearest-neighbor networks.
Figure \ref{figsub-5} illustrates the sum of squared errors between estimated edge values and true edge values, i.e., $\sum_{k=1}^{K} \sum_{i < j}\big( \overline{\Theta}^{(k)}_{ij} - \Sigma^{(k)}_{ij}\big)^2$. When the number of the total edges selected {increases} (i.e., the sparsity control parameter $w_1$ {decreases}), the error {decreases} and finally reaches a fairly low value.
Figure \ref{figsub-6}  plots the number of TP differential edges against FP differential. An edge that differs between {networks} is called a  differential edge and thus corresponds to a change-point. Numerically, we say that the $(i,j)$ edge is estimated to be differential between the $k$-th and the $(k+1)$-th networks if $|\overline{\Theta}^{(k)}_{ij} - \overline{\Theta}^{(k+1)}_{ij} | > 10^{-6}$, and we say that it is truly differential if $| \Sigma^{(k)}_{ij} - \Sigma^{(k+1)}_{ij} |> 10^{-6}$.
The number of differential edges is computed for all successive pairs of networks.
One can observe in Figure \ref{figsub-6} that the results obtained with $w_2=0.2$ have approximately 3000 TP differential edges and almost no false ones.
This suggests that the GGL model can be a suitable model to use in change-point detection of nearest-neighbor networks.

\subsubsection{Linear rate convergence}
The purpose of  this section is to demonstrate numerically the local linear convergence of the PPDNA. Specifically, we conduct experiments on two representative instances: (a)  categorical data: {\it Webtrain} with $(p,K) = (300,4) $, $(\lambda_1,\lambda_2)=(5$e-$3,5$e-$4)$; (b) time-varying data: {\it SPX500} with $(p,K) = (200,11) $, $(\lambda_1,\lambda_2)=(5$e-$4,5$e-$5)$.
Due to the lack of exact optimal solutions of these instances, we run the PPDNA until the accuracy of $10^{-10}$ is achieved and regard the resulting approximate solution as the true solution $(\overline{\Omega},\overline{\Theta},\overline{X})$. We denote
$$
d_t :=\frac{\|\overline{\Omega}_t - \overline{\Omega}\| + \|\overline{\Theta}_t - \overline{\Theta}\| + \|\overline{X}_t - \overline{X}\|}{\|\overline{\Omega}\| + \|\overline{\Theta}\| + \|\overline{X}\|},\,\,t=0,1,\dots.
$$
In Figure \ref{fig-rate},
we plot $\log_{10} d_t$ against the iteration count $t$ under two different choices of the penalty parameter $\sigma_t$: $\sigma_t$ is fixed or increased by a ratio. When $\sigma_t$ is fixed,
the solid blue line in the figure indicates that the convergence rate is almost constant. When $\sigma_{t+1}=1.3\sigma_t$, i.e., the penalty parameter is gradually increasing, the dash-dotted red line shows that the convergence rate is increasingly fast. The observation
is consistent with Theorem \ref{PPDNA-linear}, which demonstrates numerically the local linear convergence rate of the PPDNA.
We should emphasize that the impressive linear convergence rate depicted in
the solid blue curve in  Figure~\ref{fig-rate}(a) is attained with $\sigma_t$ fixed at a large value of $10^8$, whereas the slower initial convergence shown in the dash-dotted red curve is due to slowly increasing the parameter $\sigma_t$ from a small initial value of $2\times10^4$. The same remark is also applicable to Figure~\ref{fig-rate}(b).

The dependence of the linear rate of convergence on $\sigma_t$ also sheds light on the choice of $\sigma_t$ in our implementation. Basically we adaptively update $\sigma_t$
to strike a good balance in the trade-off between the convergence rate of the PPDNA  and the difficulty
in computing the Newton directions (via the CG method) in the semismooth Newton method  (Step~1 of Algorithm \ref{alg-ssn}). As the condition number of the Newton linear system in Step~1 of Algorithm \ref{alg-ssn}
is proportional to $\sigma_t$, the CG method will converge more slowly for a larger  $\sigma_t$.
Thus in our experiments, we start from a small $\sigma_0$, e.g., $\sigma_0=1$, and gradually increase $\sigma_t$ by some factor $\zeta> 1$, i.e., $\sigma_{t+1} = \zeta\sigma_t$.

\begin{figure}[H]
\centering
\begin{subfigure}[b]{.4\linewidth}
\centering
\includegraphics[width=1.0\textwidth]{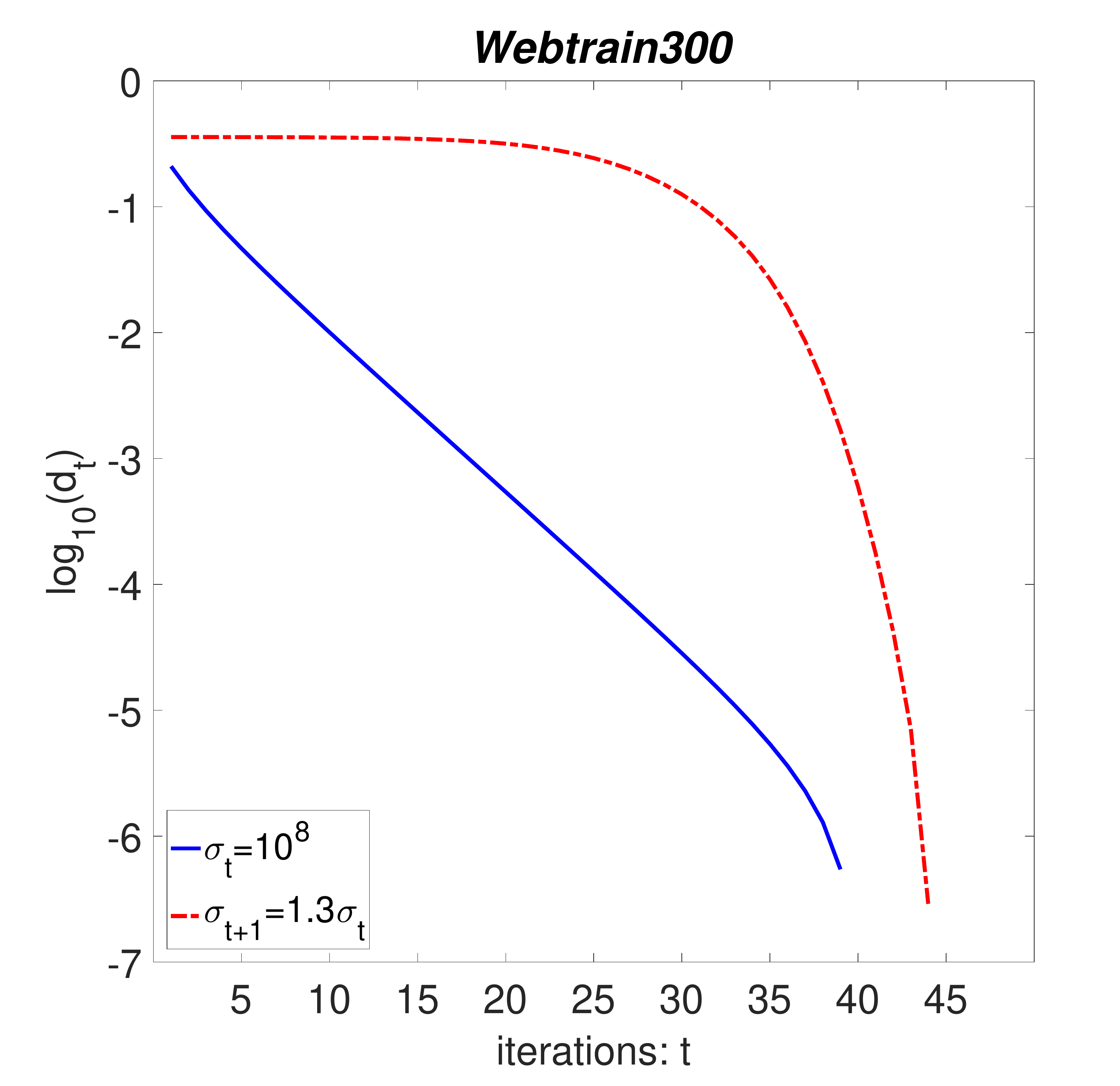}
\caption{}\label{figsub-7}
\end{subfigure}
\begin{subfigure}[b]{.4\linewidth}
\centering
\includegraphics[width=1.0\textwidth]{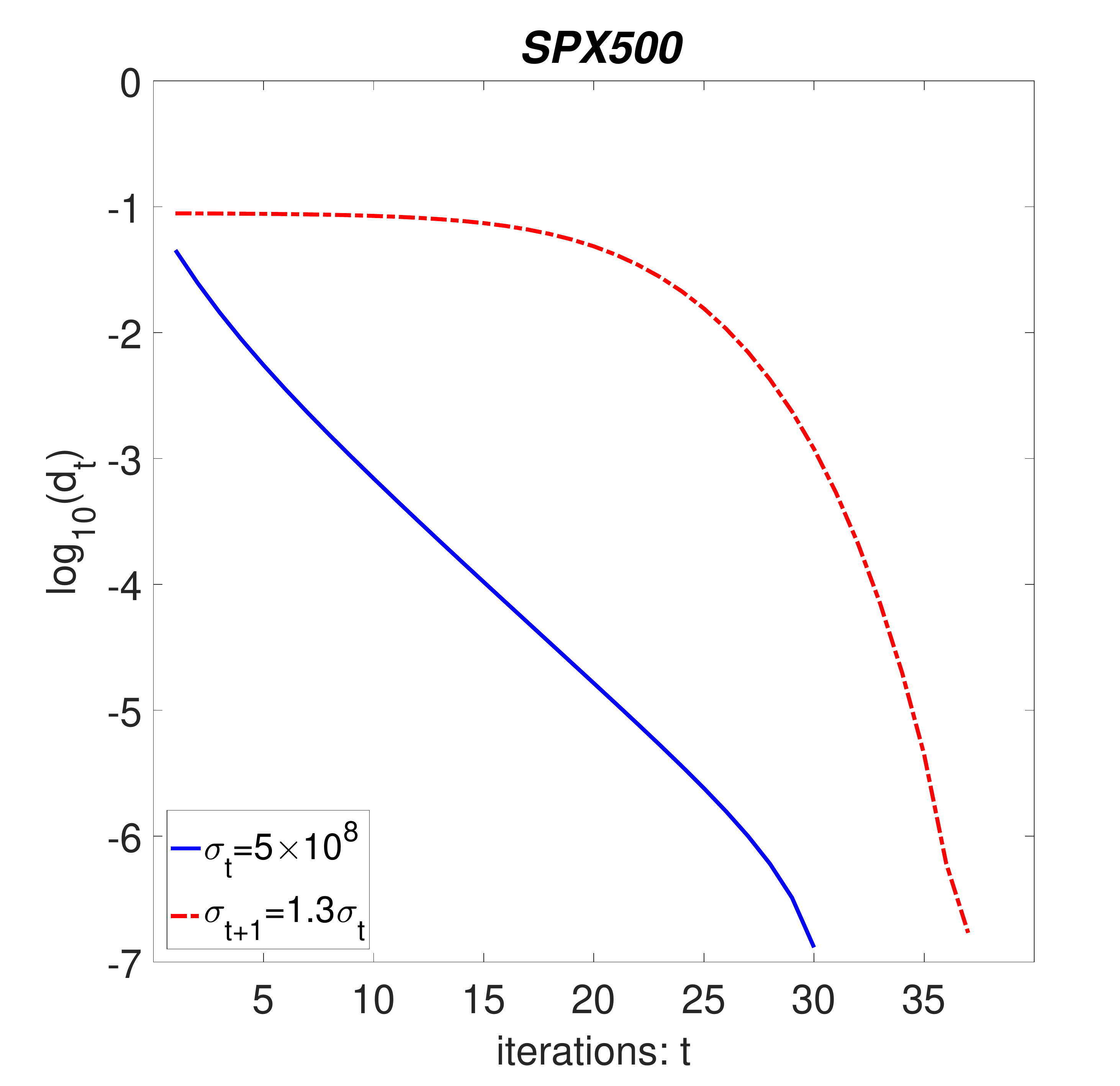}
\caption{}\label{figsub-8}
\end{subfigure}
\caption{The relative distances of the iterates generated by the PPDNA to the optimal solution. (a)  Webtrain with $(p,K) = (300,4) $, $(\lambda_1,\lambda_2)=(5$e-$3,5$e-$4)$; (b) SPX500 with $(p,K) = (200,11) $, $(\lambda_1,\lambda_2)=(5$e-$4,5$e-$5)$.} \label{fig-rate}
\end{figure}


\subsection{Comparison with ADMM and MGL}\label{sec-web}
In this section, we compare our algorithm PPDNA for solving the GGL model with the ADMM  described in \eqref{alg-ADMM} and the MGL implemented in \cite{yang2015fused}. For the tuning parameters $\lambda_1$ and $\lambda_2$, we select three pairs  for each instance that produce reasonable sparsity. In the following tables, ``P'' stands for PPDNA; ``A'' stands for ADMM; ``M'' stands for MGL.
In the column under ``Iteration'', we report the number of iterations taken by various
algorithms. In particular, for the PPDNA, we report the number of the PPA iterations taken
and the
number (within the parentheses) of semismooth Newton linear systems solved.
Let ``nnz'' denote the number of nonzero entries in the solution $\Theta$ obtained by the PPDNA using the following estimation:
$ {\rm nnz}:= \min \{k\,|\,\sum_{i=1}^{k}|\hat{x}_i| \geq 0.999\|\hat{x}\|_1\}, $
where $\hat{x}\in\mathbb{R}^{p^2K}$ is the vector obtained by sorting all elements of $\Theta$ by magnitude in a descending order. In the tables,
``density'' denotes the quantity nnz$/(p^2K)$.
The time is displayed in the format of ``hours:minutes:seconds'', and the fastest method is highlighted in red.
The error reported for the PPDNA in the tables is the relative KKT residual $\eta_P$. That of the ADMM is $\eta_A$; while the error for the MGL is $\Delta_M$.

Table \ref{table-web}  shows the comparison of three methods PPDNA, ADMM, and MGL on the university webpages data sets.  The PPDNA successfully solved all instances in Table \ref{table-web} within about one minute. For a large majority of tested instances, the PPDNA is  faster than the ADMM and the MGL.
It suggests that the PPDNA is robust and efficient for solving the GGL model applied to the university webpages data.

Table \ref{table-ng300}  presents the comparison of PPDNA, ADMM, and MGL  on the 20 newsgroups data sets. One can see clearly that the PPDNA outperforms the ADMM and the MGL for most instances in Table \ref{table-ng300}.
It demonstrates that the PPDNA can be efficient for solving the GGL model. For some difficult instances, e.g., {\it NGcomp} train $(\lambda_1,\lambda_2)=(5$e-$4,5$e-$5)$, our PPDNA took less than one minute while the MGL took more than one hour. Again, the results show that
our PPDNA is robust for solving the GGL model. The superior performance of our PPDNA can
primarily be attributed to our ability to extract and exploit the sparsity structure (in $\widehat{\partial}{\rm Prox}_{\mathcal{P}}$) within the semismooth Newton method to solve the PPA subproblems very efficiently.

Table \ref{table-spx} gives the results  on the Standard \& Poor's 500 component stock price data set {\it{SPX500}}. The table shows that the PPDNA is faster than both the ADMM and the MGL for all instances.
In addition, we find that both the PPDNA and the ADMM succeeded in solving all instances, while the MGL failed to solve one of them within three hours. This might imply that the MGL is not robust for solving the GGL model when applied to the stock price data sets. The numerical results show convincingly that our algorithm PPDNA can solve the GGL problem highly efficiently and robustly.

\begin{table}[H]\centering
\caption{Performances of PPDNA, ADMM, and MGL on university webpages data.}\label{table-web}
\setlength{\tabcolsep}{1.2mm}{\scriptsize
\begin{tabular}{lllccccccccc}
\toprule
Problem & $(\lambda_1,\lambda_2)$ & Density & \multicolumn{3}{c}{Iteration} & \multicolumn{3}{c}{Time} & \multicolumn{3}{c}{Error} \\
\cmidrule(l){4-6} \cmidrule(l){7-9} \cmidrule(l){10-12}
$(p,K)$ & &  & P & A & M & P & A & M & P & A & M\\
\midrule
                          &      (1e-02,1e-03)  & 0.016  & 16(24)  & 501  & 4  & {\color{red} 02}  & 04  & 10  & 3.2e-07  & 9.9e-07  & 2.9e-07 \\
 {\it Webtest}            &      (5e-03,5e-04)  & 0.048  & 16(25)  & 501  & 6  & {\color{red} 02}  & 04  & 13  & 3.2e-07  & 9.9e-07  & 2.6e-07 \\
 (100,4)                  &      (1e-03,1e-04)  & 0.225  & 14(22)  & 529  & 32  & {\color{red} 02}  & 04  & 56  & 2.4e-07  & 9.9e-07  & 1.0e-06 \\
&&&&&&&&&&&\\[-0.1cm]
                          &      (1e-02,1e-03)  & 0.008  & 14(24)  & 850  & 5  & {\color{red} 08}  & 25  & 37  & 7.9e-07  & 1.0e-06  & 6.7e-08 \\
 {\it Webtest}            &      (5e-03,5e-04)  & 0.026  & 14(27)  & 679  & 7  & {\color{red} 10}  & 19  & 50  & 7.9e-07  & 9.8e-07  & 4.1e-07 \\
 (200,4)                  &      (1e-03,1e-04)  & 0.163  & 13(23)  & 503  & 77  & {\color{red} 07}  & 11  & 05:57  & 2.7e-07  & 9.9e-07  & 1.6e-06 \\
&&&&&&&&&&&\\[-0.1cm]
                          &      (5e-03,5e-04)  & 0.016  & 14(29)  & 744  & 8  & {\color{red} 28}  & 33  & 02:39  & 5.6e-07  & 9.9e-07  & 5.3e-08 \\
 {\it Webtest}            &      (1e-03,1e-04)  & 0.125  & 16(32)  & 487  & 205  & 45  & {\color{red} 22}  & 21:51  & 5.9e-07  & 9.9e-07  & 1.8e-06 \\
 (300,4)                  &      (5e-04,5e-05)  & 0.256  & 14(35)  & 668  & 1128  & 55  & {\color{red} 30}  & 01:37:51  & 3.9e-07  & 9.9e-07  & 2.1e-06 \\
&&&&&&&&&&&\\[-0.1cm]
                          &      (1e-02,1e-03)  & 0.012  & 20(34)  & 1601  & 3  & {\color{red} 03}  & 11  & 08  & 1.2e-07  & 1.0e-06  & 6.0e-06 \\
 {\it Webtrain}           &      (5e-03,5e-04)  & 0.033  & 20(34)  & 1601  & 5  & {\color{red} 03}  & 12  & 04  & 1.2e-07  & 1.0e-06  & 7.4e-07 \\
 (100,4)                  &      (1e-03,1e-04)  & 0.165  & 20(34)  & 1601  & 22  & {\color{red} 04}  & 13  & 43  & 1.2e-07  & 1.0e-06  & 7.8e-06 \\
&&&&&&&&&&&\\[-0.1cm]
                          &      (5e-03,5e-04)  & 0.016  & 20(39)  & 1325  & 7  & {\color{red} 13}  & 27  & 01:18  & 1.3e-07  & 1.0e-06  & 1.3e-06 \\
 {\it Webtrain}           &      (1e-03,1e-04)  & 0.108  & 20(37)  & 1397  & 31  & {\color{red} 11}  & 31  & 02:49  & 9.9e-08  & 1.0e-06  & 5.0e-06 \\
 (200,4)                  &      (5e-04,5e-05)  & 0.219  & 20(39)  & 1397  & 88  & {\color{red} 16}  & 32  & 05:00  & 1.1e-07  & 1.0e-06  & 5.3e-06 \\
&&&&&&&&&&&\\[-0.1cm]
                          &      (5e-03,5e-04)  & 0.011  & 20(62)  & 1826  & 10  & {\color{red} 01:04}  & 01:37  & 08:35  & 2.4e-07  & 9.7e-07  & 2.7e-06 \\
 {\it Webtrain}           &      (1e-03,1e-04)  & 0.080  & 19(33)  & 1196  & 45  & {\color{red} 21}  & 52  & 09:55  & 3.4e-07  & 1.0e-06  & 6.0e-06 \\
 (300,4)                  &      (5e-04,5e-05)  & 0.177  & 19(36)  & 1196  & 134  & {\color{red} 36}  & 52  & 13:00  & 3.7e-07  & 1.0e-06  & 6.0e-06 \\
\bottomrule
\end{tabular}}
\end{table}

\begin{table}[H]
\centering
\caption{Performances of PPDNA, ADMM, and MGL  on  20 newsgroups data. }\label{table-ng300}
\setlength{\tabcolsep}{1.2mm}{\scriptsize\begin{tabular}{lllccccccccc}
\toprule
Problem & $(\lambda_1,\lambda_2)$ & Density & \multicolumn{3}{c}{Iteration} & \multicolumn{3}{c}{Time} & \multicolumn{3}{c}{Error} \\
\cmidrule(l){4-6} \cmidrule(l){7-9} \cmidrule(l){10-12}
 $(p,K)$ & &  & P & A & M & P & A & M & P & A & M\\
\midrule
 {\it NGcomp}          &        (5e-03,5e-04)  & 0.021  & 15(22)  & 509  & 31  & {\color{red} 16}  & 26  & 37:08  & 6.5e-08  & 9.9e-07  & 1.1e-06 \\
 test                  &        (1e-03,1e-04)  & 0.099  & 16(26)  & 625  & 510  & {\color{red} 32}  & 34  & 01:20:37  & 7.9e-07  & 1.0e-06  & 2.0e-06 \\
 (300,5)               &        (5e-04,5e-05)  & 0.210  & 14(24)  & 494  & 1481  & 40  & {\color{red} 27}  & 03:00:00  & 7.2e-07  & 1.0e-06  & 6.1e-06 \\
&&&&&&&&&&&\\[-0.1cm]
 {\it NGrec}           &        (5e-03,5e-04)  & 0.004  & 21(38)  & 1331  & 5  & {\color{red} 15}  & 49  & 04:04  & 8.0e-08  & 1.0e-06  & 4.9e-07 \\
 test                  &        (1e-03,1e-04)  & 0.063  & 21(39)  & 1331  & 13  & {\color{red} 20}  & 58  & 04:28  & 8.2e-08  & 1.0e-06  & 1.9e-06 \\
 (300,4)               &        (5e-04,5e-05)  & 0.143  & 20(37)  & 1331  & 36  & {\color{red} 20}  & 58  & 07:49  & 3.7e-07  & 1.0e-06  & 3.7e-06 \\
&&&&&&&&&&&\\[-0.1cm]
 {\it NGsci}           &        (5e-03,5e-04)  & 0.006  & 17(26)  & 542  & 6  & {\color{red} 14}  & 21  & 05:25  & 3.8e-07  & 1.0e-06  & 2.1e-06 \\
 test                  &        (1e-03,1e-04)  & 0.075  & 17(27)  & 553  & 21  & {\color{red} 19}  & 24  & 11:17  & 3.9e-07  & 9.8e-07  & 2.1e-06 \\
 (300,4)               &        (5e-04,5e-05)  & 0.167  & 17(31)  & 550  & 87  & 25  & {\color{red} 24}  & 17:13  & 5.0e-07  & 9.9e-07  & 2.6e-06 \\
&&&&&&&&&&&\\[-0.1cm]
 {\it NGtalk}          &        (5e-03,5e-04)  & 0.026  & 15(25)  & 482  & 16  & {\color{red} 24}  & 26  & 26:08  & 9.2e-08  & 9.6e-07  & 4.1e-07 \\
 test                  &        (1e-03,1e-04)  & 0.115  & 12(23)  & 278  & 81  & 20  & {\color{red} 13}  & 13:39  & 1.2e-07  & 9.9e-07  & 1.1e-06 \\
 (300,3)               &        (5e-04,5e-05)  & 0.240  & 11(22)  & 286  & 337  & 25  & {\color{red} 13}  & 40:28  & 9.4e-08  & 9.7e-07  & 2.2e-06 \\
&&&&&&&&&&&\\[-0.1cm]
 {\it NGcomp}          &        (5e-03,5e-04)  & 0.016  & 16(31)  & 1150  & 13  & {\color{red} 33}  & 57  & 14:58  & 1.2e-07  & 1.0e-06  & 5.6e-08 \\
 train                 &        (1e-03,1e-04)  & 0.080  & 15(31)  & 1153  & 172  & {\color{red} 35}  & 01:04  & 40:11  & 4.6e-07  & 1.0e-06  & 1.9e-06 \\
 (300,5)               &        (5e-04,5e-05)  & 0.153  & 15(30)  & 1216  & 574  & {\color{red} 33}  & 01:07  & 01:12:51  & 4.4e-07  & 1.0e-06  & 1.8e-06 \\
&&&&&&&&&&&\\[-0.1cm]
 {\it NGrec}           &        (5e-03,5e-04)  & 0.005  & 19(35)  & 1519  & 5  & {\color{red} 22}  & 52  & 02:36  & 1.4e-07  & 1.0e-06  & 3.9e-07 \\
 train                 &        (1e-03,1e-04)  & 0.068  & 18(37)  & 1500  & 16  & {\color{red} 31}  & 01:06  & 09:45  & 2.6e-07  & 1.0e-06  & 4.8e-06 \\
 (300,4)               &        (5e-04,5e-05)  & 0.124  & 18(35)  & 1542  & 48  & {\color{red} 28}  & 01:07  & 09:02  & 2.9e-07  & 1.0e-06  & 5.4e-06 \\
&&&&&&&&&&&\\[-0.1cm]
 {\it NGsci}           &        (5e-03,5e-04)  & 0.011  & 17(30)  & 1387  & 10  & {\color{red} 21}  & 54  & 10:00  & 1.7e-07  & 1.0e-06  & 3.5e-08 \\
 train                 &        (1e-03,1e-04)  & 0.086  & 16(32)  & 1389  & 40  & {\color{red} 32}  & 01:01  & 08:57  & 5.1e-07  & 1.0e-06  & 2.6e-06 \\
 (300,4)               &        (5e-04,5e-05)  & 0.152  & 16(32)  & 965  & 206  & {\color{red} 37}  & 42  & 18:10  & 4.1e-07  & 9.9e-07  & 2.9e-06 \\
&&&&&&&&&&&\\[-0.1cm]
 {\it NGtalk}          &        (5e-03,5e-04)  & 0.026  & 18(32)  & 2445  & 13  & {\color{red} 26}  & 01:41  & 13:24  & 2.6e-07  & 1.0e-06  & 1.9e-06 \\
 train                 &        (1e-03,1e-04)  & 0.103  & 17(32)  & 2448  & 52  & {\color{red} 19}  & 01:46  & 13:26  & 1.4e-07  & 1.0e-06  & 4.4e-06 \\
 (300,3)               &        (5e-04,5e-05)  & 0.204  & 17(38)  & 2385  & 213  & {\color{red} 28}  & 01:45  & 19:43  & 7.2e-08  & 1.0e-06  & 4.9e-06 \\
\bottomrule
\end{tabular}}
\end{table}
\begin{table}[H]
\centering
\caption{Performances of PPDNA, ADMM, and MGL  on stock price data. }\label{table-spx}
\setlength{\tabcolsep}{1.2mm}{\scriptsize\begin{tabular}{lllccccccccc}
\toprule
Problem & $(\lambda_1,\lambda_2)$ & Density & \multicolumn{3}{c}{Iteration} & \multicolumn{3}{c}{Time} & \multicolumn{3}{c}{Error} \\
\cmidrule(l){4-6} \cmidrule(l){7-9} \cmidrule(l){10-12}
 $(p,K)$ & &  & P & A & M & P & A & M & P & A & M\\
\midrule
                        &      (1e-04,1e-05)  & 0.039  & 22(33)  & 3644  & 6  & {\color{red} 04}  & 22  & 28  & 1.1e-07  & 1.0e-06  & 9.8e-07 \\
{\it SPX500}            &      (5e-05,5e-06)  & 0.138  & 22(35)  & 3646  & 8  & {\color{red} 05}  & 23  & 25  & 1.5e-07  & 1.0e-06  & 8.4e-06 \\
(100,3)                 &      (2e-05,2e-06)  & 0.238  & 23(43)  & 2056  & 18  & {\color{red} 08}  & 13  & 08  & 4.4e-07  & 1.0e-06  & 2.5e-05 \\
&&&&&&&&&&&\\[-0.1cm]
                        &      (1e-04,1e-05)  & 0.025  & 24(31)  & 1409  & 8  & {\color{red} 12}  & 23  & 01:20  & 8.8e-08  & 1.0e-06  & 9.4e-06 \\
{\it SPX500}            &      (5e-05,5e-06)  & 0.084  & 21(28)  & 1239  & 17  & {\color{red} 14}  & 20  & 04:23  & 9.4e-08  & 1.0e-06  & 9.0e-06 \\
(200,3)                 &      (2e-05,2e-06)  & 0.150  & 20(38)  & 1363  & 32  & {\color{red} 18}  & 23  & 03:28  & 1.4e-07  & 9.9e-07  & 2.0e-05 \\
&&&&&&&&&&&\\[-0.1cm]
                        &      (5e-04,5e-05)  & 0.030  & 22(29)  & 3701  & 11  & {\color{red} 12}  & 01:01  & 02:20  & 4.3e-08  & 1.0e-06  & 5.4e-06 \\
{\it SPX500}            &      (1e-04,1e-05)  & 0.127  & 22(30)  & 3722  & 105  & {\color{red} 18}  & 01:21  & 02:34  & 9.3e-08  & 1.0e-06  & 7.6e-06 \\
(100,11)                &      (5e-05,5e-06)  & 0.206  & 22(30)  & 2925  & 393  & {\color{red} 21}  & 01:06  & 09:12  & 7.8e-07  & 1.0e-06  & 7.8e-06 \\
&&&&&&&&&&&\\[-0.1cm]
                        &      (5e-04,5e-05)  & 0.018  & 19(24)  & 1096  & 28  & {\color{red} 31}  & 53  & 36:07  & 8.4e-07  & 1.0e-06  & 4.1e-06 \\
{\it SPX500}            &      (1e-04,1e-05)  & 0.082  & 19(24)  & 1125  & 481  & {\color{red} 49}  & 01:08  & 01:27:17  & 7.9e-07  & 1.0e-06  & 5.1e-06 \\
(200,11)                &      (5e-05,5e-06)  & 0.140  & 19(27)  & 1101  & 1258  & {\color{red} 01:05}  & 01:08  & 03:00:00  & 6.1e-07  & 1.0e-06  & 1.7e-05 \\
\bottomrule
\end{tabular}}
\end{table}

\section{Concluding remarks}\label{sec:conclusion}
In this paper, we have taken advantage of the ideas proposed in \cite{li2017efficiently,Zhang2018efficient} and implemented a proximal point dual Newton algorithm~(PPDNA) to the primal formulation of the group graphical Lasso problems. From a theoretical standpoint, we have shown that the PPDNA is globally convergent and the sequence of primal and dual iterates is Q-linearly convergent, although the group graphical Lasso regularizer is non-polyhedral. The robustness and  superior numerical efficiency of the PPDNA are convincingly demonstrated in various numerical experiments. Therefore, we can firmly conclude that the PPDNA is not only a fast method with nice theoretical guarantees, but also a numerically efficient method for solving the group graphical Lasso problems with multiple precision matrices.

\end{document}